\documentclass{amsart}

\usepackage[utf8]{inputenc}
\usepackage[T1]{fontenc}
\usepackage{times}
\usepackage{graphicx}
\usepackage{enumitem}
\usepackage{amssymb}
\usepackage{nicefrac}
\usepackage{stmaryrd}
\usepackage{bussproofs}
\usepackage{tikz-cd}
\usepackage{hyperref}

\makeatletter
\@namedef{subjclassname@2020}{%
	\textup{2020} Mathematics Subject Classification}
\makeatother

\def\fCenter{\mbox{\ $\vdash$\ }}
\newenvironment{bprooftree}
{\leavevmode\hbox\bgroup}
{\DisplayProof\egroup}

\DeclareMathOperator{\dom}{dom}
\DeclareMathOperator{\lAC}{\lp-AC^{0,0}}
\newcommand{\lto}{\multimap}
\DeclareMathOperator{\minus}{\dot{-}}
\newcommand{\zerosAfter}[2]{\overline{#1, #2}}

\newcommand{\N}{\mathbb{N}}

\DeclareMathOperator{\IQC}{IQC}

\DeclareMathOperator{\RCA}{RCA}
\DeclareMathOperator{\EL}{EL}
\DeclareMathOperator{\ELexalpha}{EL^{\exists \alpha a}_0}
\DeclareMathOperator{\MP}{MP}

\DeclareMathOperator{\lp}{\ell}

\DeclareMathOperator{\EAPAomegal}{E-APA^{\omega}_{\lp}}
\DeclareMathOperator{\EAPAomegaeq}{E-APA^{\omega}_{\doteq}}
\DeclareMathOperator{\LL}{LL}
\DeclareMathOperator{\ELPAomegal}{E-LPA^{\omega}_{\lp}}
\DeclareMathOperator{\ELPAomegaeq}{E-LPA^{\omega}_{\doteq}}
\DeclareMathOperator{\EPAomega}{E-PA^{\omega}}
\DeclareMathOperator{\QFACnil}{QF-AC^{0,0}}
\DeclareMathOperator{\wit}{wit}
\DeclareMathOperator{\lth}{lth}
\DeclareMathOperator{\con}{con}

\newcommand{\forallL}{\forall^{\lp}}
\newcommand{\existsL}{\exists^{\lp}}
\newcommand{\existsLeps}{\exists^{\lp}_{\epsilon}}

\newcommand{\with}{\mathbin{\&}}
\newcommand{\Par}{\mathbin{\rotatebox[origin=c]{180}{\&}}}
\newcommand{\place}{\mathalpha{\square}}

\newcommand{\pair}[2]{j(#1, #2)}

\DeclareMathOperator{\sg}{sg}

\newcommand{\Aqf}{A_{qf}}
\newcommand{\Aat}{A_{at}}
\newcommand{\Anl}{A_{nl}}

\newcommand{\Sem}[2]{\llbracket #1 \rrbracket_{#2}}

\newcommand{\Real}[3]{\lvert #1 \rvert^{#2}_{#3}}

\newcommand{\tupl}[1]{\boldsymbol{#1}}
\newcommand{\emL}[1]{#1^{\bullet}}
\newcommand{\comp}[1]{\vec{#1}}
\newcommand{\halts}[1]{#1 \mathclose{\downarrow}}

\newcommand{\period}{\text{.}}
\newcommand{\comma}{\text{,}}

\theoremstyle{plain}
\newtheorem{theorem}{Theorem}[section]
\newtheorem{lemma}[theorem]{Lemma}
\newtheorem{corollary}[theorem]{Corollary}

\theoremstyle{remark}
\newtheorem{remark}[theorem]{Remark}

\theoremstyle{definition}
\newtheorem{definition}[theorem]{Definition}

\title[Characterization of Weihrauch reducibility]{The characterization of\\Weihrauch reducibility in systems containing $\EPAomega + \QFACnil$}
\author{Patrick Uftring}
\address{\parbox{\linewidth}{Department of Mathematics\\
Technische Universit\"at Darmstadt\\
Schlossgartenstra\ss e 7\\
64289 Darmstadt, Germany}}
\email{patrick\_juergen.uftring@stud.tu-darmstadt.de}
\thanks{\copyright~The Association for Symbolic Logic 2020. This is the accepted version of a paper published in \emph{The Journal of Symbolic Logic} 86(1), pp.~224--261, \href{https://doi.org/10.1017/jsl.2020.53}{DOI:10.1017/jsl.2020.53}.}

\begin{document}

\begin{abstract}
	We characterize Weihrauch reducibility in $\EPAomega + \QFACnil$ and all systems containing it by the provability in a linear variant of the same calculus using modifications of G\"odel's Dialectica interpretation that incorporate ideas from linear logic, nonstandard arithmetic, higher-order computability, and phase semantics.
\end{abstract}

\maketitle

\section{Introduction}

Many theorems in mathematics can be formulated using partial multi-valued functions, we call them \emph{problems}. The arguments on which problems are defined are \emph{instances}, and the values they map to are \emph{solutions}. For example, Weak K\H{o}nig's Lemma can be interpreted as a problem whose instances are infinite binary trees and whose solutions are infinite paths in such trees.

Naturally, one wants to compare the strength of different problems. A fruitful approach is the following: Let $P$ and $Q$ be two problems. If we can map any instance $x$ of $Q$ to an instance $f(x)$ of $P$, and any solution $y$ of $f(x)$ to a solution $g(x, y)$ of $x$, then $Q$ \emph{reduces} to $P$.

\begin{center}
	\begin{tikzcd}
	\dom(P) \ni f(x) \arrow[r] & y \in P(f(x)) \arrow[d, "g"]\\
	\dom(Q) \ni x \arrow[u, "f"] \arrow[r] & g(x, y) \in Q(x)
	\end{tikzcd}
\end{center}

In the late 1980s, Klaus Weihrauch defined this notion of reducibility in two technical reports (cf. \cite[p.~5]{Weihrauch} and \cite[p.~5]{Weihrauch2}) where the mappings $f$ and $g$ are required to be continuous and problems are defined on Baire space. The modern formulation of \emph{Weihrauch reducibility} (cf. Definition 2.1 in \cite[p.~147]{weihrauchDegrees}), which was originally called \emph{computable reducibility}, uses \emph{computable} mappings $f$ and $g$, extending reductions to problems on more general, so-called \emph{represented} spaces. The Weihrauch reduction from $Q$ to $P$ expresses that we need exactly one instance of $P$ to solve $Q$.

The proposed goal is to find a characterization of Weihrauch reducibility that is of the form
\begin{equation*}
	\forall x (A(x) \to \exists y B(x, y)) \to \forall u (C(u) \to \exists v D(u, v))
\end{equation*}
where $A(x)$ and $C(u)$ are predicates expressing that $x$ and $u$ are instances of $Q$ and $P$, respectively. Moreover, $B(x, y)$ and $D(u, v)$ hold if and only if $y$ is a solution for $x$ and $v$ is a solution for $u$, respectively.

\subsection*{First results about a proof-theoretic characterization}

In his 2017 article \cite{kuyper2017}\\\emph{On Weihrauch reducibility and intuitionistic reverse mathematics}, Rutger Kuyper proposed a characterization of compositional Weihrauch reducibility formalized in $\RCA_0$, the base-system of reverse mathematics with induction restricted to $\Sigma^0_1$-formulas, using the intuitionistic fragment $\EL_0$ of $\RCA_0$ together with Markov's principle $\MP$. \emph{Compositionally Weihrauch reducing} $Q$ to $P$ means that there is a fixed natural number $n$ such that $Q$ reduces to the composition of $n$ copies of $P$, i.e.~we use $P$ $n$-many times in order to yield $Q$. Moreover, he also proposed a characterization of true, i.e.~non-compositional Weihrauch reducibility using an affine fragment $(\EL_0 + \MP)^{\exists \alpha a}$ of $\EL_0 + \MP$. At the end of this paper, we will discuss some issues with these results:

\begin{enumerate}[label=\arabic*)]
	\item There are problems $P$ and $Q$ for which we can prove $P \to Q$ such that $Q$ does not reduce to a finite composition of $P$. However, the original formulation of Kuyper's first result claims that $Q$ compositionally Weihrauch-reduces to $P$.
	\item The definition of $\EL_0$ used in the results does not provide disjunction. While this connective can be simulated both in the classical variant $\RCA_0$ and the version $\EL$ with full induction, $\EL_0$ without disjunction is not able to do so, and is therefore not as strong as intuitionistic logic.
	\item The second result about characterizing true Weihrauch reducibility already fails for a rather simple class of problems, namely those where the set of instances is unrestricted.
\end{enumerate}

The question if there exists such a connection between logic and Weihrauch reducibility remains. Since G\"odel's Dialectica interpretation and Weihrauch reducibility seem to share some similarities as already noticed by other researchers, e.g.~Fernando Ferreira, this is where we made our own attempts at finding a characterization of Weihrauch reducibility in $\EPAomega + \QFACnil$ and related calculi.

\subsection*{Using G\"odel's Dialectica Interpretation to extract Weihrauch programs}

For the characterizing calculus, we start off with \emph{linear logic}, a refinement of classical logic published in 1987 by Jean-Yves Girard (cf.~\cite{girard1987}). In order to introduce arithmetic to this system, we embed the axioms of \emph{extensional Peano arithmetic in all finite types} $\EPAomega$ (cf.~1.6 in \cite{troelstra1973} and \cite[pp.~46--51]{kohlenbach2008}) using a variation of Girard's embedding from intuitionistic into linear logic. In order to differentiate quantifiers that are used within the formalization of problems from those that describe the existence of a solution for each instance, we borrow the standard predicate \emph{st} from \emph{nonstandard arithmetic} together with its axioms from work due to Benno van den Berg, Eyvind Briseid, and Pavol Safarik about a Dialectica interpretation variant for nonstandard arithmetic (cf.~\cite{nonstandard}). We rename this predicate to the \emph{linear predicate} $\lp$ as both its behavior and computational content are different within the context of linear logic. Finally, we add some axioms that lower the restrictions imposed by linear logic on weakening, contraction, and searches in well-behaved situations. The implication of problems is then expressed in the following way
\begin{equation}\tag{$*$}
	\forallL x (A(x) \lto \existsL y B(x, y)) \lto \forallL u (C(u) \lto \existsL v D(u, v))
\end{equation}
where the superscript $\lp$ restricts $x$, $y$, $u$, and $v$ to \emph{linear} values, meaning that reasoning about properties of these values is restricted by linear logic.

The main result of this paper (Theorem \ref{thm:secondCharacterization}) will be that this implication characterizes a natural formalization of Weihrauch reducibility in $\EPAomega + \QFACnil$ and all systems containing it using \emph{associates} (cf.~\cite{kleene}, \cite{kreisel}, and 2.6 in \cite{troelstra1973}) in order to express computability.

For the proof of this result, we apply a modified version of G\"odel's Dialectica interpretation to ($*$). This modification is based both on functional interpretations for linear logic by Valeria de Paiva, Masaru Shirahata, and Paolo Oliva (cf.~\cite{dePaiva89}, \cite{dePaiva1991}, \cite{shirahata}, \cite{olivaClassical}, \cite{lmcsFerreiraOliva}, and \cite{springerFerreiraOliva}), and on a functional interpretation of the previously mentioned nonstandard arithmetic (cf.~\cite{nonstandard}) for the interpretation of the linear predicate. We refine this variant of the Dialectica interpretation even further by incorporating computability into the interpretation of the linear predicate. After these first steps, we arrive at the following characterization of Weihrauch reducibility (Theorem \ref{thm:firstCharacterization})
\begin{equation}\tag{$\dagger$}
\forallL x (A(x) \lto \existsLeps y B(x, y)) \lto \forallL u (C(u) \lto \existsLeps v D(u, v))
\end{equation}
where the parameter $\epsilon$ is a free variable whose value can let both $\existsLeps$ behave as normal linear existential quantifiers or remove them together with the subformula they quantify over from the statement, which we use to prove that the first program of the Weihrauch reduction always terminates.

\subsection*{Using phase semantics to ensure termination}
In order to land at the more symmetric and elegant version ($*$), we apply a second modified version of G\"odel's Dialectica interpretation beforehand: This functional interpretation applies tags to all linear predicates. We can choose the values of these tags for both negative occurrences of the predicate. In order to extract the tags of both positive occurrences, we apply a simplified version of Girard's phase semantics (cf.~\cite[pp.~17--28]{girard1987}) to the terms and formulas produced by this functional interpretation. It turns out that both existential quantifiers share the same tag, which is different from the tag shared by both universal quantifiers. By substituting the tags for the existential quantifiers, we can convert the formula ($*$) into the characterization ($\dagger$) and reuse the result that we will have proven in the first part of this paper.

\subsection*{Further literature}

There exists further work about the application of proof theory in form of modified realizability to Weihrauch reducibility, namely by Jeffry L. Hirst and Carl Mummert (cf.~\cite{hirst2019}), and Makoto Fujiwara (cf.~\cite{fuji2020}). These results are mostly about direct applications to reverse mathematics for restricted classes of problems.

\section{Linear Peano Arithmetic}

\begin{definition}[$\LL$] Like classical logic, linear logic ($\LL$) starts with a set of function symbols from which terms are defined inductively. Moreover, for a set of predicate symbols, the formulas of linear logic are defined as follows:
	\begin{itemize}
		\item $1$, $0$, $\top$, $\bot$,
		\item $P_n(t_1, \dots, t_n)$, $P_n^{\bot}(t_1, \dots, t_n)$,
		\item $A \otimes B$, $A \Par B$, $A \oplus B$, $A \with B$,
		\item $!A$, $?A$,
		\item $\exists x A$, and $\forall x A$
	\end{itemize}
	where $P_n$ is an $n$-ary predicate symbol and $t_1$, \dots, $t_n$ are terms.
	We define the involution $A^{\bot}$ as abbreviation recursively:
	\begin{itemize}
		\item $1^{\bot} :\equiv \bot$, $0^{\bot} :\equiv \top$, $\top^{\bot} :\equiv 0$, $\bot^{\bot} :\equiv 1$,
		\item $(P_n(t_1, \dots, t_n))^{\bot} :\equiv P_n^{\bot}(t_1, \dots, t_n)$, $(P_n^{\bot}(t_1, \dots, t_n))^{\bot} :\equiv P_n(t_1, \dots, t_n)$
		\item $(A \otimes B)^{\bot} :\equiv A^{\bot} \Par B^{\bot}$, $(A \Par B)^{\bot} :\equiv A^{\bot} \otimes B^{\bot}$, $(A \with B)^{\bot} :\equiv A^{\bot} \oplus B^{\bot}$, $(A \oplus B)^{\bot} :\equiv A^{\bot} \with B^{\bot}$,
		\item $(!A)^{\bot} :\equiv ?A^{\bot}$, $(?A)^{\bot} :\equiv !A^{\bot}$,
		\item $(\exists x A)^{\bot} :\equiv \forall x A^{\bot}$, $(\forall x A)^{\bot} :\equiv \exists x A^{\bot}$.
	\end{itemize}
	We abbreviate $A \lto B :\equiv A^{\bot} \Par B$ and introduce the axioms and rules of linear logic:
	
	\begin{itemize}
		
		\item Identity and structure:
		\begin{center}
			\begin{bprooftree}
				\Axiom$\fCenter A, A^{\bot}$
			\end{bprooftree}(id)
			\hskip 1.5em
			\begin{bprooftree}
				\Axiom$\fCenter \Gamma, A$
				\Axiom$\fCenter A^{\bot}, \Delta$
				\RightLabel{(cut)}
				\BinaryInf$\fCenter \Gamma, \Delta$
			\end{bprooftree}
			\hskip 1.5em
			\begin{bprooftree}
				\Axiom$\fCenter \Gamma$
				\RightLabel{(per)}
				\UnaryInf$\fCenter \Gamma'$
			\end{bprooftree}
		\end{center}
		where $\Gamma'$ is a permutation of $\Gamma$.
		
		\item Multiplicatives:
		
		\begin{center}
			\begin{bprooftree}
				\Axiom$\fCenter \Gamma, A$
				\Axiom$\fCenter \Delta, B$
				\RightLabel{($\otimes$)}
				\BinaryInf$\fCenter \Gamma, \Delta, A \otimes B$
			\end{bprooftree}
			\hskip 1.5em
			\begin{bprooftree}
				\Axiom$\fCenter \Gamma, A, B$
				\RightLabel{($\Par$)}
				\UnaryInf$\fCenter \Gamma, A \Par B$
			\end{bprooftree}
			\\
			\vskip 1.5em
			\begin{bprooftree}
				\Axiom$\fCenter 1$
			\end{bprooftree}($1$)
			\hskip 1.5em
			\begin{bprooftree}
				\Axiom$\fCenter \Gamma$
				\RightLabel{($\bot$)}
				\UnaryInf$\fCenter \Gamma, \bot$
			\end{bprooftree}
		\end{center}
		
		\item Additives:
		\begin{center}
			\begin{bprooftree}
				\Axiom$\fCenter \Gamma, A$
				\Axiom$\fCenter \Gamma, B$
				\RightLabel{($\with$)}
				\BinaryInf$\fCenter \Gamma, A \with B$
			\end{bprooftree}
			\hskip 1.5em
			\begin{bprooftree}
				\Axiom$\fCenter \Gamma, \top$
			\end{bprooftree}($\top$)
			\\
			\vskip 1.5em
			\begin{bprooftree}
				\Axiom$\fCenter \Gamma, A$
				\RightLabel{($\oplus_1$)}
				\UnaryInf$\fCenter \Gamma, A \oplus B$
			\end{bprooftree}
			\hskip 1.5em
			\begin{bprooftree}
				\Axiom$\fCenter \Gamma, B$
				\RightLabel{($\oplus_2$)}
				\UnaryInf$\fCenter \Gamma, A \oplus B$
			\end{bprooftree}
		\end{center}
		
		\item Modalities:
		\begin{center}
			\begin{bprooftree}
				\Axiom$\fCenter ?\Gamma, A$
				\RightLabel{($!$)}
				\UnaryInf$\fCenter ?\Gamma, !A$
			\end{bprooftree}
			\hskip 1.5em
			\begin{bprooftree}
				\Axiom$\fCenter \Gamma, A$
				\RightLabel{(d$?$)}
				\UnaryInf$\fCenter \Gamma, ?A$
			\end{bprooftree}
			\\
			\vskip 1.5em
			\begin{bprooftree}
				\Axiom$\fCenter \Gamma$
				\RightLabel{(w$?$)}
				\UnaryInf$\fCenter \Gamma, ?A$
			\end{bprooftree}
			\hskip 1.5em
			\begin{bprooftree}
				\Axiom$\fCenter \Gamma, ?A, ?A$
				\RightLabel{(c$?$)}
				\UnaryInf$\fCenter \Gamma, ?A$
			\end{bprooftree}
		\end{center}
		where (d$?$), (w$?$), and (c$?$) stand for ``dereliction'', ``weakening'', and ``contraction''.
		
		\item Quantifiers:
		\begin{center}
			\begin{bprooftree}
				\Axiom$\fCenter \Gamma, A$
				\RightLabel{($\forall$)}
				\UnaryInf$\fCenter \Gamma, \forall x A$
			\end{bprooftree}
			\hskip 1.5em
			\begin{bprooftree}
				\Axiom$\fCenter \Gamma, A[\nicefrac{t}{x}]$
				\RightLabel{($\exists$)}
				\UnaryInf$\fCenter \Gamma, \exists x A$
			\end{bprooftree}
		\end{center}
		where in ($\forall$) $x$ is not free in $\Gamma$, and in ($\exists$) $t$ is free for $x$ in $A$.
		
	\end{itemize}
\end{definition}

The following embedding of classical logic into linear logic is inspired by Girard's embedding in \cite{girard1987}.

\begin{definition}[Embedding of classical logic into linear logic]\label{Definition:Embedding}
	For any formula $A$ of classical logic, we define an embedding $\emL{A}$ into linear logic.
	\begin{alignat*}{3}
		&\emL{A} &&:\equiv A \text{ where $A$ is atomic.}\\
		\intertext{Let $\emL{A}$ and $\emL{B}$ already be defined for $A$ and $B$.}
		&\emL{(\lnot A)} &&:\equiv (\emL{A})^{\bot}\comma\\
		&\emL{(A \land B)} &&:\equiv \emL{A} \otimes \emL{B}\comma\\
		&\emL{(A \lor B)} &&:\equiv \emL{A} \Par \emL{B}\comma\\
		&\emL{(A \to B)} &&:\equiv \emL{A} \lto \emL{B}\comma\\
		&\emL{(\exists x^{\tau} A)} &&:\equiv \exists x^{\tau} \emL{A}\comma\\
		&\emL{(\forall x^{\tau} A)} &&:\equiv \forall x^{\tau} \emL{A}\period
	\end{alignat*}
\end{definition}

In the next step, we will integrate \emph{extensional Peano arithmetic in all finite types} $\EPAomega$ into this linear calculus. A definition and introduction to this system can be found in \cite[pp.~41--52]{kohlenbach2008}. We also use some term definitions, including the signum $\sg$ and maximum $\max$, as well as addition, terms for comparison, and a special subtraction (written as $x \minus y$) that returns zero instead in the case of negative results (cf. Definition 3.26 in \cite[p.~45]{kohlenbach2008}). Since classical linear logic is given by a sequent calculus, we will use a version of $\EPAomega$ that is in the form of a Gentzen-style sequent calculus. We will also often add the following axiom of quantifier-free choice (cf.~\cite[p.~53]{kohlenbach2008})
\begin{equation*}
	\QFACnil: \forall x^0 \exists y^0 \Aqf(x, y) \to \exists Y^1 \forall x^0 \Aqf(x, Yx)
\end{equation*}
where $\Aqf$ is quantifier-free.

\begin{definition}[$\ELPAomegal$]
	The language of \emph{extensional linear Peano arithmetic in all finite types with linear predicate} ($\ELPAomegal$) is defined like $\LL$, but we add some new predicates and rules. First, the terms of this language are those of $\EPAomega$. We also add its types and typed quantifiers. Moreover, we add equality for numbers and the \emph{linear predicate} $\lp$, i.e.~we add the predicates $s^0 =_0 {s'}^{0}$ and $\lp_{\tau}(t^{\tau})$ where both $s$ and $s'$ are terms of type $0$, and $t$ is a term of type $\tau$. We call a formula \emph{nonlinear} if and only if it does not contain $\oplus$, $\with$, or $\lp$.
	We may omit the type if we can infer it from the context. Higher order equality is, just like in $\EPAomega$, an abbreviation:
	\begin{equation*}
		s =_{\tau\rho} t :\equiv \forall x^{\rho} (sx =_{\tau} tx)
	\end{equation*}
	where $x$ does not occur in $s$ or $t$. Moreover, we write $s \neq_{\tau} t$ as abbreviation for $(s =_{\tau} t)^{\bot}$.
	
	Now, we use Definition \ref{Definition:Embedding} in order to transport all axioms of $\EPAomega$ to $\ELPAomegal$, i.e.~for each axiom $A$ of $\EPAomega$, we add $\emL{A}$ to $\ELPAomegal$. This includes axioms for equality, the successor term, projectors, combinators, recursors, and induction.
	
	Finally, we have some special axioms concerning the linear predicate:
	\begin{center}
		\begin{bprooftree}
			\Axiom$\fCenter \lp_{\tau}(t^{\tau})$
		\end{bprooftree}($\lp$)
		\hskip 0.7em
		\begin{bprooftree}
			\Axiom$\fCenter \Anl^{\bot}, !\Anl$
		\end{bprooftree}($!_2$)
		\hskip 0.7em
		\begin{bprooftree}
			\Axiom$\fCenter \Gamma, \lp^{\bot}_{\tau}(t), \lp^{\bot}_{\tau}(t)$
			\RightLabel{($\lp$-con)}
			\UnaryInf$\fCenter \Gamma, \lp^{\bot}_{\tau}(t)$
		\end{bprooftree}
		\\
		\vskip 1.5em
		\begin{bprooftree}
			\Axiom$\fCenter \lp_{\tau \rho}^{\bot}(t), \lp_{\rho}^{\bot}(r), \lp_{\tau}(tr)$
		\end{bprooftree}($\lp$-app)
		\vskip 1em
		\begin{bprooftree}
			\Axiom$\fCenter (\forall x^0 \exists y^0 \alpha x y =_0 0)^{\bot}, \exists Y^1 (\forall x^0 (\alpha x (Yx) =_0 0) \otimes !(\lp(\alpha) \lto \lp(Y)))$
		\end{bprooftree}($\lAC$)
	\end{center}
	where $t$ in ($\lp$) is a closed term, and $\Anl$ in ($!_2$) is nonlinear and may not contain $0$ or $\top$.
\end{definition}

\begin{remark}
	The axioms ($\lp$) and ($\lp$-app) are taken directly from those for the standard predicate in the nonstandard calculus defined by \cite{nonstandard}.
	The rule ($\lp$-con) allows us to contract the negated linear predicate. This enables us to use the linear predicate multiple times during an argument even if we only assume it once. The other direction does not hold: We are not allowed to forget any of its occurrences, in general.
	Similarly, (!$_2$) lifts all restrictions on weakening and contraction for nonlinear formulas. This makes it possible to copy any result from $\EPAomega$ directly to $\ELPAomegal$ via our embedding.
	
	Finally, ($\lAC$) basically is an axiom of choice for terms $\alpha^{0(0)(0)}$ that we use for the computation of witnesses, i.e.~if for every $x^0$ we can find a $y^0$ with $\alpha x y =_0 0$ then there exists a witness $Y$ outputting such a $y$ for every $x$ that we give to it. Additionally, if we can compute $\alpha$ and its underlying program is visible (this is a possible interpretation of what the linear predicate indicates), then we can compute $Y$ as well. This axiom allows the linear predicate to witness searches and therefore computability.
	
	For a practical use one might want to add extensionality for linear predicates, but such an axiom is not needed for the characterization result.
\end{remark}

\begin{lemma}\label{lem:substitute}
	Assume that we can prove $\vdash P^{\bot}, Q$ in $\ELPAomegal$ and let $A$ be a formula of the same language containing a placeholder $\$$ only positively, i.e.~after unraveling all abbreviations and involutions, $A$ does not contain the subformula $\$^{\bot}$. Then $\ELPAomegal$ proves $\vdash (A[\nicefrac{P}{\$}])^{\bot}, A[\nicefrac{Q}{\$}]$.
\end{lemma}

\begin{proof}
	We prove this by induction on the length of $A$.
	\begin{itemize}
		\item $\$$: This follows directly from $\vdash P^{\bot}, Q$.
		\item $A$ atomic: Here, we need to show $\vdash A^{\bot}, A$, which is an instance of (id).
	\end{itemize}
	Now, assume that we have already shown the result for $A$ and $B$, i.e.~$\vdash (A[\nicefrac{P}{\$}])^{\bot}, A[\nicefrac{Q}{\$}]$ and $\vdash (B[\nicefrac{P}{\$}])^{\bot}, B[\nicefrac{Q}{\$}]$ hold:
	\begin{itemize}
		\item $\exists x^{\tau} A$ or $\forall x^{\tau} A$: First, we apply ($\exists$) to $A[\nicefrac{Q}{\$}]$ (resp. to $(A[\nicefrac{P}{\$}])^{\bot}$) and then ($\forall$) to $(A[\nicefrac{P}{\$}])^{\bot}$ (resp. to $A[\nicefrac{Q}{\$}]$), which yields $\vdash (\exists / \forall x^{\tau} A[\nicefrac{P}{\$}])^{\bot}, \exists / \forall x^{\tau} A[\nicefrac{Q}{\$}]$.
		\item $!A$ or $?A$: First, we apply ($d?$) to $(A[\nicefrac{P}{\$}])^{\bot}$(resp. to $A[\nicefrac{Q}{\$}]$) and then ($!$) to $A[\nicefrac{Q}{\$}]$ (resp. to $(A[\nicefrac{P}{\$}])^{\bot}$), which yields $\vdash (!/? x^{\tau} A[\nicefrac{P}{\$}])^{\bot}, !/? x^{\tau} A[\nicefrac{Q}{\$}]$.
		\item $A \otimes B$ or $A \Par B$: First, we apply ($\otimes$) to $A[\nicefrac{Q}{\$}]$ and $B[\nicefrac{Q}{\$}]$ (resp. to $(A[\nicefrac{P}{\$}])^{\bot}$ and $(B[\nicefrac{P}{\$}])^{\bot}$) and then ($\Par$) to $(A[\nicefrac{P}{\$}])^{\bot}$ and $(B[\nicefrac{P}{\$}])^{\bot}$ (resp. to $A[\nicefrac{Q}{\$}]$ and $B[\nicefrac{Q}{\$}]$), which yields $\vdash ((A \otimes B / A \Par B)[\nicefrac{P}{\$}])^{\bot}, (A \otimes B / A \Par B)[\nicefrac{Q}{\$}]$.
		\item $A \with B$ or $A \oplus B$: First, we apply ($\oplus_1$) to $(A[\nicefrac{P}{\$}])^{\bot}$ and ($\oplus_2$) to $(B[\nicefrac{P}{\$}])^{\bot}$ (resp. ($\oplus_1$) to $A[\nicefrac{P}{\$}]$ and ($\oplus_2$) to $B[\nicefrac{P}{\$}]$), and then ($\with$) to $A[\nicefrac{Q}{\$}]$ and $B[\nicefrac{Q}{\$}]$ (resp. to $(A[\nicefrac{P}{\$}])^{\bot}$ and $(B[\nicefrac{P}{\$}])^{\bot}$), which yields $\vdash ((A \with B / A \oplus B)[\nicefrac{P}{\$}])^{\bot}, (A \with B / A \oplus B)[\nicefrac{Q}{\$}]$.
	\end{itemize}
\end{proof}

\begin{corollary}\label{cor:substituteTensorPar}
	Let $\$$ be an only positively occurring placeholder in $A$ like before. Assume that we can prove $\vdash B^{\bot}, 1$, then we can show:
	\begin{enumerate}[label=\arabic*)]
		\item $\vdash (A[\nicefrac{B \otimes C}{\$}])^{\bot}, A[\nicefrac{C}{\$}]$,
		\item $\vdash (A[\nicefrac{C}{\$}])^{\bot}, A[\nicefrac{B^{\bot} \Par C}{\$}]$.
	\end{enumerate}
	Assume that we can prove $\vdash B$, then we can show:
	\begin{enumerate}[resume, label=\arabic*)]
		\item $\vdash (A[\nicefrac{C}{\$}])^{\bot}, A[\nicefrac{B \otimes C}{\$}]$,
		\item $\vdash (A[\nicefrac{B^{\bot} \Par C}{\$}])^{\bot}, A[\nicefrac{C}{\$}]$.
	\end{enumerate}
	Similar statements hold for permutations, i.e. $C \otimes B$, etc.
\end{corollary}

\begin{proof}
	We use Lemma \ref{lem:substitute}. For 1) we have to show $\vdash (B \otimes C)^{\bot}, C$:
	\begin{prooftree}
		\Axiom$\fCenter B^{\bot}, 1$
		\AxiomC{}
		\RightLabel{(id)}
		\UnaryInf$\fCenter C^{\bot}, C$
		\RightLabel{($\bot$)}
		\UnaryInf$\fCenter \bot, C^{\bot}, C$
		\RightLabel{(cut)}
		\BinaryInf$\fCenter B^{\bot}, C^{\bot}, C$
		\RightLabel{($\Par$)}
		\UnaryInf$\fCenter (B \otimes C)^{\bot}, C$
	\end{prooftree}
	The proof of $\vdash C^{\bot}, (B^{\bot} \Par C)$ for 2) is almost identical. For 3) we have to show $\vdash C^{\bot}, B \otimes C$:
	\begin{prooftree}
		\AxiomC{}
		\RightLabel{(id)}
		\UnaryInf$\fCenter C^{\bot}, C$
		\Axiom$\fCenter B$
		\RightLabel{($\otimes$)}
		\BinaryInf$\fCenter C^{\bot}, B \otimes C$
	\end{prooftree}
	The proof of $\vdash (B^{\bot} \Par C)^{\bot}, C$ for 4) is almost identical.
\end{proof}

\begin{lemma}\label{lem:embedPeanoIntoLinear}\mbox{}\\
	If we can prove $\Gamma \vdash \Delta$ in $\EPAomega + \QFACnil$, then we can also derive $\vdash (\emL{\Gamma})^{\bot}, \emL{\Delta}$ in $\ELPAomegal$.
\end{lemma}

\begin{proof}
	The proof works by induction on the length of the derivation and is very straightforward. The induction starts with the axioms $A$ of $\EPAomega + \Gamma$, whose embeddings $\emL{A}$ are, by definition, axioms in $\ELPAomegal + \emL{\Gamma}$. The most interesting part are the contraction and weakening rules, which we will now show in detail. We start with the step for the (left-)contraction of $\Gamma, A, A \vdash \Delta$:
	
	\begin{prooftree}
		\Axiom$\fCenter (\emL{\Gamma})^{\bot}, (\emL{A})^{\bot}, (\emL{A})^{\bot}, \emL{\Delta}$
		\RightLabel{(d?)}
		\UnaryInf$\fCenter (\emL{\Gamma})^{\bot}, ?(\emL{A})^{\bot}, ?(\emL{A})^{\bot}, \emL{\Delta}$
		\RightLabel{(c?)}
		\UnaryInf$\fCenter (\emL{\Gamma})^{\bot}, ?(\emL{A})^{\bot}, \emL{\Delta}$
		\AxiomC{}
		\RightLabel{($!_2$*)}
		\UnaryInf$\fCenter (\emL{A})^{\bot}, !\emL{A}$
		\RightLabel{(cut)}
		\BinaryInf$\fCenter (\emL{\Gamma})^{\bot}, (\emL{A})^{\bot}, \emL{\Delta}$
	\end{prooftree}
	where we used in ($!_2$*) that our embedding maps to nonlinear formulas. The result is the embedded version of $\Gamma, A \vdash \Delta$. A similar proof works for right-contraction. Now, we take a look at the (left-)weakening of $\Gamma \vdash \Delta$:
	\begin{prooftree}
		\Axiom$\fCenter (\emL{\Gamma})^{\bot}, \emL{\Delta}$
		\RightLabel{(w?)}
		\UnaryInf$\fCenter (\emL{\Gamma})^{\bot}, ?(\emL{A})^{\bot}, \emL{\Delta}$
		\AxiomC{}
		\RightLabel{($!_2$)}
		\UnaryInf$\fCenter (\emL{A})^{\bot}, !\emL{A}$
		\RightLabel{(cut)}
		\BinaryInf$\fCenter (\emL{\Gamma})^{\bot}, (\emL{A})^{\bot}, \emL{\Delta}$
	\end{prooftree}
	The result is the embedded version of $\Gamma, A \vdash \Delta$. An analogous proof works for right-weakening. We will use similar reasoning for applying contraction and weakening to nonlinear formulas throughout this article.
	
	Finally, for $\QFACnil$, we know that quantifier-free formulas can be expressed using terms in $\EPAomega$ and that this fact can be proven without using the quantifier-free axiom itself (cf.~Proposition 3.17 \cite[p.~50]{kohlenbach2008}). Therefore, we only have to build the axiom of choice for terms. We yield such an axiom
	\begin{equation*}
		\fCenter (\forall x^0 \exists y^0 \alpha x y =_0 0)^{\bot}, \exists Y^1 \forall x^0 (\alpha x (Yx) =_0 0)
	\end{equation*}
	if we combine $\vdash (!(\lp(\alpha) \lto \lp(Y)))^{\bot}, 1$, which can be proven using ($1$) and (w?), with Corollary \ref{cor:substituteTensorPar} and ($\lAC$).
\end{proof}

The following system will only be used internally and is not needed for understanding the result itself:
\begin{definition}[$\ELPAomegaeq$]
	The language of $\ELPAomegaeq$ is defined like that of $\ELPAomegal$, but we add the predicates $s^0 \doteq_0 t^0$ and $\place(t^0)$ for terms $s$ and $t$ of type $0$, and read $s \doteq_{\tau\rho} t$ as abbreviation for $\forall x^{\rho} (sx \doteq_{\tau} tx)$. Also, we write $s \not\doteq_0 t$ as abbreviation for $(s \doteq_0 t)^{\bot}$.
	Moreover, we further restrict (!$_2$) to formulas that do contain neither of our new predicates and add the following rules
	\begin{center}
		\begin{bprooftree}
			\Axiom$\fCenter !Sx \not\doteq_0 0$
		\end{bprooftree}(dot-succ)\\
		\vskip 1.5em
		\begin{bprooftree}
			\Axiom$\fCenter s =_{\tau} t$
			\RightLabel{(dot-eq$_1$)}
			\UnaryInf$\fCenter !s \doteq_{\tau} t$
		\end{bprooftree}
		\hskip 1.5em
		\begin{bprooftree}
			\Axiom$\fCenter (!k \not\doteq_0 0)^{\bot}, !\sg(k) \doteq_0 1$
		\end{bprooftree}(dot-eq$_2$)
		\\
		\vskip 1.5em
		\begin{bprooftree}
			\Axiom$\fCenter (!x \doteq_{\tau} y)^{\bot}, A^{\bot}[\nicefrac{x}{z}], A[\nicefrac{y}{z}]$
		\end{bprooftree}(dot-sub)
		\vskip 1.5em
		\begin{bprooftree}
			\Axiom$\fCenter \place(\sg(0))$
		\end{bprooftree}(tag-0)
		\hskip 1em
		\begin{bprooftree}
			\Axiom$\fCenter \Gamma, \place^{\bot}(t), \place^{\bot}(t)$
			\RightLabel{(tag-con)}
			\UnaryInf$\fCenter \Gamma, \place^{\bot}(t)$
		\end{bprooftree}
		\vskip 1em
		\begin{bprooftree}
			\Axiom$\fCenter \place^{\bot}(\sg(x)), \place^{\bot}(\sg(y)), \place(\sg(x + y))$
		\end{bprooftree}(tag-app)
	\end{center}\vskip 1em
	where $x$ and $y$ are free for $z$ in $A$.
\end{definition}

\begin{remark}
	Since we want to have a substitution axiom for the tagging predicate, we need to introduce the stronger dot-equality: If we allowed the normal equality to substitute values within tags, it would not be possible to define appropriate phase semantics in Section \ref{sec:Phase}.
\end{remark}

\begin{corollary}
	If a sequent can be proven in $\ELPAomegal$, then it can also be proven in $\ELPAomegaeq$.
\end{corollary}
\begin{proof}
	$\ELPAomegaeq$ is an extension of $\ELPAomegal$. While we further restrict (!$_2$) to formulas that do not contain dot-equality or the tagging predicate, this is not a problem because both predicates do not exist in $\ELPAomegal$.
\end{proof}

\begin{lemma}\label{lem:embedLinearIntoPeano}
	If we can prove a sequent $\vdash \Delta$ in $\ELPAomegaeq + \emL{\Gamma}$ where $\Gamma$ are formulas of $\EPAomega$, then we can also prove the sequent $\vdash \Delta'$ in $\EPAomega + \QFACnil + \Gamma$ where $\Delta'$ results from replacing $1$, $\top$, $\lp(t^{\tau})$, and $\place(t^0)$ by $0 =_0 0$, replacing $0$ and $\bot$ by $1 =_0 0$, replacing $\otimes$ and $\with$ by $\land$, and replacing $\Par$ and $\oplus$ by $\lor$. Moreover, we substitute $\doteq$ with the regular $=$, the linear negation $(.)^{\bot}$ by the classical $\lnot$, and omit any occurrences of the modalities $!$ or $?$.
\end{lemma}

\begin{proof}
	We do not give the detailed the proof since it is known that linear logic is a refinement of classical logic. The axioms we took from $\EPAomega$ to $\ELPAomegal$ can easily be transported back with these translations. A similar argument works for $\emL{\Gamma}$ and $\Gamma$. The axioms and rules for the linear and tagging predicate, except for ($\lAC$), hold trivially as their substitutions are tautologies. The axioms and rules for the dot-equality also hold true for equality in extensional Peano arithmetic. Finally, the translation of ($\lAC$) can be proven using $\QFACnil$.
\end{proof}

\begin{definition}[Affine logic]
	If we, instead of (w?), add the following rule of weakening to the system $\ELPAomegal$ or $\ELPAomegaeq$
	\begin{center}
		\begin{bprooftree}
			\Axiom$\fCenter \Gamma$
			\RightLabel{(w),}
			\UnaryInf$\fCenter \Gamma, A$
		\end{bprooftree}
	\end{center}
	we call the resulting \emph{affine} calculus $\EAPAomegal$ or $\EAPAomegaeq$.
\end{definition}

\begin{remark}
	The results of Lemma \ref{lem:substitute} and Corollary \ref{cor:substituteTensorPar} also work for $\ELPAomegaeq$, $\EAPAomegal$, and $\EAPAomegaeq$.
\end{remark}

\section{Computing in Higher Types}

Since Weihrauch reductions consist of computable functions, we have to implement computability into our functional interpretation: We do this by recursively defining a relation on arbitrary types that lets us represent higher-order computations using the terms of $\EPAomega$. Finally, we will show that for low types these terms are equivalent to \emph{associates} (cf.~\cite{kleene}, \cite{kreisel}, and 2.6 in \cite{troelstra1973}) that we use in our formalization of Weihrauch reducibility.

\subsection*{Pairs and sequences}

The following definitions and results about pairs and sequences are taken from \cite{troelstra1973} and \cite{kohlenbach2008}:

\begin{definition}[Pairs]
	We define $j^{0(0)(0)}$, $j_1^{0(0)}$, and $j_2^{0(0)}$ (cf. \cite[p.~23]{troelstra1973}) to be terms such that we can prove the following in $\EPAomega$:
	\begin{itemize}
		\item $\fCenter j_1 (j x^0 y^0) =_0 x$,
		\item $\fCenter j_2 (j x^0 y^0) =_0 y$,
		\item $\fCenter j x^0 y^0 \geq_0 y$.
	\end{itemize}
	Let $x$, $y$, and $z$ be variables of type $0$, we may write $j(x, y)$ as a more readable variant of $jxy$. Moreover, we will use $z_l$ and $z_r$ as shorthands for $j_1 z$ and $j_2 z$, respectively.
\end{definition}

\begin{definition}[Finite sequences]	
	We can code finite sequences in $\EPAomega$ (cf. \cite[p.~24]{troelstra1973}). We write
	\begin{equation*}
	x := \langle x_1, \dots, x_n \rangle
	\end{equation*}
	for a finite sequence $x$ of length $n$.  
	
	\begin{itemize}
		\item We define $\lth^{0(0)}$ to be a term such that $\lth x$ tells us the length of arbitrary sequences $x$.
		\item We may write $x_k$ for the $k$-th element of $x$ starting at $1$.
		\item We can concatenate two finite sequences $x^0$ and $y^0$, or one finite sequence $x^0$ and one infinite sequence $y^1$ using $x * y$.
		\item Finally, we abbreviate $\hat{x} :\equiv \langle x \rangle$.
	\end{itemize}
\end{definition}

\begin{definition}[Initial segments]
	Let $x^1$ be an infinite sequence, and $n^0$ a natural number. We define the abbreviations (cf. \cite[p.~202]{kohlenbach2008})
	\begin{flalign*}
	&&\overline{x}n &:= \langle x0, \dots, x(n-1) \rangle&\\
	\text{and } &&\zerosAfter{x}{n} &:= \lambda k^0. \left\{
	\begin{aligned}
	xk\phantom{0} &\text{ if $k <_0 n$,}\\
	0\phantom{xk} &\text{otherwise.}
	\end{aligned}\right.&
	\end{flalign*}
\end{definition}

\begin{definition}[Hereditarily computable types]
	For every type $\tau$ we define the hereditarily computable type $\comp{\tau}$\index{$\comp{\tau}$} inductively:
	\begin{align*}
	\comp{0} &:\equiv 1\comma\\
	\comp{(\tau \rho)} &:\equiv \comp{\tau}\comp{\rho}\period
	\end{align*}
\end{definition}

\begin{definition}[Constructing terms]
	For every type $\tau$ and pair of terms $s^{\comp{\tau}}$ and $t^\tau$, we define in $\EPAomega$ by induction what it means for $s$ to \emph{construct} $t$, which we write as $\con_{\tau}(s,t)$\index{c@$\con_{\tau}$}:
	\begin{align*}
	&\con_0 (s^1, t^0) &&:\equiv \exists x^0 (sx \neq_0 0) \land \forall x^0 (sx \neq_0 0 \to sx =_0 t + 1)\comma\\
	&\con_{\tau\rho} (s^{\comp{\tau\rho}}, t^{\tau\rho}) &&:\equiv \forall x^{\comp{\rho}}, y^\rho (\con_{\rho} (x,y) \to \con_{\tau} (sx, ty))\period
	\end{align*}
	We simply write $\con(s,t)$ if the type is clear from context.
\end{definition}

\begin{definition}[Constructing terms for small types]
	Let $s^0$ and $t^1$ be terms. We define
	\begin{flalign*}\index{$\tilde{x}$}
	&&\tilde{s} &:\equiv \lambda k^0. s + 1&\\
	\text{and }&&\tilde{t} &:\equiv \lambda z^1, k^0. \left\{
	\begin{aligned}
	t(zk \minus 1) + 1\phantom{0} &\text{if $zk \neq_0 0$,}\\
	0\phantom{t(zk \minus 1) + 1} &\text{otherwise.}
	\end{aligned}\right.
	\end{flalign*}
	The definition therefore depends on the type of the term.
\end{definition}

\begin{lemma}\label{closedTermsAreCp}
	For every term $t^\tau$ where all free variables are of type $0$ or $1$, there exists a term $s^{\comp{\tau}}$ with the same free variables such that $\EPAomega$ proves $\con(s, t)$. Additionally, $\con(\tilde{s}, s)$ and $\con(\tilde{t}, t)$ hold for terms $s^0$ and $t^1$, respectively.
\end{lemma}

\begin{remark}
	In the following proof and later in the paper, we may write tuples of terms $t_1, \dots, t_n$ in bold: $\tupl{t}$. Furthermore, the application of two tuples $\tupl{s}$ and $\tupl{t}$ is defined as follows: $\tupl{st} :\equiv \tupl{s}_1 \tupl{t}, \dots, \tupl{s}_m \tupl{t}$ where $m$ is the length of $\tupl{s}$.
\end{remark}

\begin{proof}
	We show this by induction on the length of $t$:
	\begin{itemize}
		\item $s^0$: We show $\con(\tilde{s}, s)$: Both $\tilde{s}k \neq_0 0$ and $\tilde{s}k=_0 s + 1$ hold for all $k^0$.
		\item $t^1$: We show $\con(\tilde{t}, t)$: Suppose $\con(z, n)$ holds for a sequence $z^1$ and a number $n^0$. This implies that there exists some number $k^0$ with $zk = n+1$, and that every member of $z$ is equal to either zero or $n+1$. Therefore, $\tilde{t}zk = t(zk \minus 1) + 1 = tn+1$ holds and every member of $\tilde{t}z$ is equal to either zero or $tn+1$. We conclude $\con(\tilde{t}z, tn)$ and finally $\con(\tilde{t}, t)$.
		\item $\Pi_{\rho, \tau}$: We simply take $\Pi_{\comp{\rho}, \comp{\tau}}$. Suppose $\con(x, u)$ and $\con(y, v)$ hold for $x^{\comp{\rho}}$, $y^{\comp{\tau}}$, $u^{\rho}$ and $v^{\tau}$. The former one implies $\con(\Pi_{\comp{\rho}, \comp{\tau}} x y, \Pi_{\rho, \tau} u v)$ and therefore $\con(\Pi_{\comp{\rho}, \comp{\tau}}, \Pi_{\rho, \tau})$.
		\item $\Sigma_{\delta, \rho, \tau}$: Again, we simply take $\Sigma_{\comp{\delta}, \comp{\rho}, \comp{\tau}}$. Suppose $\con(x, u)$, $\con(y, v)$, and $\con(z, w)$ hold for $x^{\comp{\tau \rho \delta}}$, $y^{\comp{\rho \delta}}$, $z^{\comp{\delta}}$, $u^{\tau \rho \delta}$, $v^{\rho \delta}$, and $w^{\delta}$. We can show $\con(xz, uw)$, $\con(yz, vw)$, and therefore, $\con((xz)(yz), (uw)(vw))$, which is equivalent to the statement $\con(\Sigma_{\comp{\delta},\comp{\rho}, \comp{\tau}} xyz, \Sigma_{\delta, \rho, \tau} uvw)$. We conclude $\con(\Sigma_{\comp{\delta}, \comp{\rho}, \comp{\tau}}, \Sigma_{\delta, \rho, \tau})$.
		\item $(R_i)_{\tupl{\rho}}$: Let $k$ be the length of $\rho$. Let $\tupl{y}^{\tupl{\rho}}$ and
		$\tupl{z}$ be tuples where $\tupl{z}$ consists of $z_j^{\rho_j 0 \tupl{\rho}}$ for $j^0$ with $1 \leq j \leq k$, and assume that both $\con(\tupl{y}', \tupl{y})$ and $\con(\tupl{z}', \tupl{z})$ hold for $\tupl{y}'$ and $\tupl{z}'$ of appropriate type.\footnote{Both $\con(\tupl{y}', \tupl{y})$ and $\con(\tupl{z}', \tupl{z})$ are abbreviations for $\con(y'_i, y_i)$ and $\con(z'_i, z_i)$ for all $i$ with $1 \leq i \leq k$, respectively.} We define terms $(T_i)_{\tupl{\rho}}$ similar to $(R_i)_{\tupl{\rho}}$:
		\begin{equation*}
		(\tupl{T}_{\tupl{\rho}}): \left\{
		\begin{aligned}
		(T_i)_{\rho} 0 \tupl{y}' \tupl{z}' &=_{\comp{\rho}_i} y'_i\\
		(T_i)_{\rho} (Sx^0) \tupl{y}' \tupl{z}' &=_{\comp{\rho}_i} z'_i (\tupl{T}_{\tupl{\rho}} x \tupl{y}' \tupl{z}') \tilde{x}\period
		\end{aligned}
		\right.
		\end{equation*}
		Notice: The only difference between $(\tupl{T}_{\tupl{\rho}})$ and $(\tupl{R}_{\tupl{\rho}})$ lies in the second equation where we do not give $z'_i$ the number $x$ of the previous recursion step directly, but indirectly as the sequence $\tilde{x}$ constructing this number, instead. We do this because $z'_i$ is a constructing term and therefore expects itself constructing terms (such as $\tilde{x}$) for numbers where $z_i$ expects actual numbers (such as $x$).

		First, we show for all natural numbers $i$ with $1 \leq i \leq k$ and $n$ that $\con((T_i)_{\tupl{\rho}} n, (R_i)_{\tupl{\rho}} n)$ holds by induction on $n$.  We have $\con(y'_i, y_i)$ and therefore $\con((T_i)_{\tupl{\rho}} 0 \tupl{y}' \tupl{z}', (R_i)_{\tupl{\rho}} 0 \tupl{y} \tupl{z})$, which gives us $\con((T_i)_{\tupl{\rho}} 0, (R_i)_{\tupl{\rho}} 0)$. We assume that our hypothesis holds for $n$. Now we can conclude that this yields
		\begin{equation*}
		\con(z'_i ((T_1)_{\tupl{\rho}} n \tupl{y}' \tupl{z}') \dots ((T_k)_{\tupl{\rho}} n \tupl{y}' \tupl{z}')\tilde{n}, z_i ((R_1)_{\tupl{\rho}} n \tupl{y} \tupl{z}) \dots ((R_k)_{\tupl{\rho}} n \tupl{y} \tupl{z})n)\comma
		\end{equation*}
		which in turn implies $\con((T_i)_{\tupl{\rho}} (Sn) \tupl{y}' \tupl{z}', (R_i)_{\tupl{\rho}} (Sn) \tupl{y} \tupl{z})$, and therefore\\ $\con((T_i)_{\tupl{\rho}} (Sn), (R_i)_{\tupl{\rho}} (Sn))$.
		
		Now, let $\tupl{\tau}$ be a tuple of types such that $\rho_i = 0 \tau_l \dots \tau_1$ holds where $l$ is the length of $\tupl{\tau}$. We define the term
		\begin{equation*}
		s:= \lambda {x'}^1, {\tupl{y}'}^{\comp{\tupl{\rho}}}, \tupl{z}', {\tupl{u}'}^{\comp{\tupl{\tau}}}, v^0. \left\{
		\begin{aligned}
		(T_i)_{\tupl{\rho}} (x'v_l \minus 1) \tupl{y} \tupl{z} \tupl{u} v_r\phantom{\text{0}} &\text{if $x' v_l \neq_0 0$,}\\
		\text{0}\phantom{(T_i)_{\tupl{\rho}} (xv_l \minus 1) \tupl{y} \tupl{z} \tupl{u} v_r} &\text{otherwise}
		\end{aligned}\right.
		\end{equation*}
		where $\tupl{y}$, $\tupl{y}'$, $\tupl{z}$, and $\tupl{z}'$ have the same types as before, and $v_l$ and $v_r$ are the components of $v$ if we interpret it as a pair of natural numbers. Additionally to $\con(\tupl{y}', \tupl{y})$ and $\con(\tupl{z}', \tupl{z})$, assume now that $\con(x', x)$ and $\con(\tupl{u}', \tupl{u})$ hold for variables $x$ and $\tupl{u}$ of matching types. Let $v_l$ be a natural number such that $x' v_l \neq_0 0$ holds. Then $\con((T_i)_{\tupl{\rho}} (x'v_l \minus 1), (R_i)_{\tupl{\rho}} (x'v_l \minus 1))$. This implies the existence of some natural number $v_r$ with $(T_i)_{\tupl{\rho}} (x'v_l \minus 1) \tupl{y}' \tupl{z}' \tupl{u}' v_r \neq_0 0$. We conclude that there exists some $v := \pair{v_l}{v_r}$ with $sx' \tupl{y}' \tupl{z}' \tupl{u}' v \neq_0 0$. Furthermore, for all $v^0$ with $x' v_l \neq_0 0$, we have $x' v_l \minus 1 =_0 x$, and therefore that $s x' \tupl{y}' \tupl{z}' \tupl{u}' v =_0 (T_i)_{\tupl{\rho}} x \tupl{y}' \tupl{z}' \tupl{u}' v_r$ is equal to zero or $(R_i)_{\tupl{\rho}} x \tupl{y} \tupl{z} \tupl{u} + 1$. We conclude $\con(s x' \tupl{y}' \tupl{z}' \tupl{u}', (R_i)_{\tupl{\rho}} x \tupl{y} \tupl{z} \tupl{u})$, and finally $\con(s, (R_i)_{\tupl{\rho}})$.
	\end{itemize}
	
	Assume that $s^{\tau \rho}$ and $t^{\rho}$ are terms where all free variables are of type $0$ or $1$ with $\con(s', s)$ and $\con(t', t)$ for some terms ${s'}^{\comp{\tau \rho}}$ and ${t'}^{\comp{\rho}}$ with the same free variables. By simply using its hereditary definition, we conclude $\con(s't', st)$. This way, we can inductively show the claim for all considered terms.
\end{proof}

\begin{definition}[Associates, \cite{kleene}, \cite{kreisel}, and \cite{troelstra1973}]\label{associateConvention}\index{associate}\index{$\halts{t \cdot x}$}\index{w@$\wit$}
	Let $t^1$ be a term. We abbreviate
	\begin{align*}
	\halts{t \cdot z^1} &:\equiv \forall u^0 \exists v^0 (t (\hat{u}*\overline{z}v) \neq_0 0)\\
	\wit(t \cdot z^1, x^1) &:\equiv \forall u^0 \exists v^0 (t (\hat{u}*\overline{z}v) =_0 xu + 1 \land \forall w <_0 v (t (\hat{u}*\overline{z}w) =_0 0))\period
	\end{align*}
	In other words: $\halts{t \cdot z}$ holds if for all $u^0$ there exists some finite initial segment $z'$ of $z$ with $t(\hat{u} * z') \neq_0 0$. The predicate $\wit(t \cdot z, x)$ holds if and only if for all $u^0$ the least initial segment $z'$ of $z$ with $t(\hat{u} * z') \neq_0 0$ lets us compute the value $xu =_0 t(\hat{u} * z') - 1$.
\end{definition}

The following two lemmas enable us to translate between constructing terms and associates. The first lemma extracts an associate out of an constructing term:

\begin{lemma}\label{le:associateFacts}
	We can prove the following in $\EPAomega + \QFACnil$:
	\begin{enumerate}[label=\alph*)]
		\item\label{leit:associateFactsContinuity} For all closed terms $t^{0(1)}$ there exists some closed term $s^{0(1)}$, a so-called \emph{modulus of pointwise continuity}\index{modulus of pointwise continuity}, such that the following holds:
		\begin{equation*}
		\bigwedge_{0 \leq_0 k <_0 sx^1} x k =_0 y^1 k \to tx =_0 ty \comma
		\end{equation*}
		\item\label{leit:associateFactsSearch} For every closed term $t^{0(1)}$ there exists some closed term $s^{0(0)}$ with both
		\begin{equation*}
		\bigwedge_{0 \leq_0 k <_0 n} x^1 k =_0 y^1 k \land s(\overline{x} n) =_0 0 \to tx =_0 t(\overline{x, n})
		\end{equation*}
		and $\forall x^1 \exists n^0 \forall m \geq_0 n \, (s(\overline{x}m) =_0 0)$.
		\item\label{leit:associateFactsReplaceConWithWit} For all closed terms $t^{\comp{1}(1)}$ there exists some closed term $s^1$ such that
		\begin{flalign*}
		&&&\con(ty^1, z^1) \to \halts{s \cdot y}&\\
		&\text{and }&&\con(ty^1, z^1) \land \wit(s \cdot y, w^1) \to z =_1 w&
		\end{flalign*}
		hold.
	\end{enumerate}
\end{lemma}

\begin{proof}\mbox{}
	\begin{enumerate}[label=\alph*)]
		\item This follows from Theorem 2.7.8 in \cite[p.~158]{troelstra1973}.
		\item Let ${t'}^{0(1)}$ be the modulus of continuity from \ref{leit:associateFactsContinuity} for $t$. We define the term
		\begin{equation*}
		s := \lambda k^0. t'(r(k)) \dot{-} \lth(k)
		\end{equation*}
		where $r^{1(0)}$ is the term that maps finite sequences represented as natural numbers $k$ to infinite sequences that start like $k$ and continue with zeros. Assume that $s(\overline{x}^1 n^0) =_0 0$ holds. Then $n$, which is the length of $\overline{x}n$, is greater or equal to $t'(r(\overline{x}n)) =_0 t'(\zerosAfter{x}{n})$. Since $x$ and $\zerosAfter{x}{n}$ coincide at the first $n$ positions, this means that $tx =_0 t(\overline{x, n})$ holds. Now, we want to show $\forall x^1 \exists n^0 \forall m \geq_0 n (s(\overline{x}m) =_0 0)$. For this purpose, let $t''$ be the modulus of continuity from \ref{leit:associateFactsContinuity} for $t'$. Then, $t' x =_0 t'(\overline{x, n})$ holds for all natural numbers $n \geq_0 t'' x$. If we define $n := \max(t'x, t''x)$, we have the following for $m \geq_0 n$: $t'(\overline{x, m}) =_0 t'x \leq_0 n \leq_0 m =_0 \lth(\overline{x}m)$. This implies $s(\overline{x}m) =_0 0$ for all $m \geq_0 n$.
		\item Let $r^{1(0)}$ be a term that maps finite sequences to infinite sequences like before. We define
		\begin{alignat*}{3}
		&s'' &&:= \lambda k^0. (\lth(k) \minus 1)_l \comma&\\
		&s' &&:= \lambda x^1. t(\lambda k^0. x(k+2)) \widetilde{(x1)} (x0) \comma&\\
		&s &&:= \lambda k^0. \left\{
		\begin{aligned}
		s'(r(\widehat{s''k} * k))\phantom{0} &\text{if $q(\widehat{s''k} * k) =_0 0$,}\\
		0\phantom{s'(r(\widehat{s''k} * k))} &\text{otherwise}
		\end{aligned}\right.&
		\end{alignat*}
		where $q^1$ is the term with the property from \ref{leit:associateFactsSearch} for $s'$. Let $y^1$ be arbitrary and assume that there exists some $z^1$ with $\con(ty,z)$. First, we show that $\halts{s \cdot y}$ holds: Let $u^0$ be arbitrary, then there exists some $n^0$ with $ty\tilde{u}n \neq_0 0$ since $\con(ty\tilde{u}, zu)$ holds. By definition of $q$, there exists some $m^0$ such that $q(\overline{\hat{n} * \hat{u} * y} (m' + 2)) =_0 0$ holds for all $m' \geq_0 m$. Let $k := \hat{u}*\overline{y}(\pair{n}{m})$. We want to show $sk =_0 ty\tilde{u}n$ in order to argue that $sk$ is nonzero. But first, we prove $s''(\hat{u} * \overline{y}(\pair{n}{m})) =_0 n$:
		\begin{align}
		&s''(\hat{u} * \overline{y}(\pair{n}{m}))\notag\\
		=_0\ &s''(\overline{\hat{u} * y}(\pair{n}{m} + 1))\notag\\
		=_0\ &(\lth(\overline{\hat{u} * y}(\pair{n}{m} + 1)) \minus 1)_l\notag\\
		=_0\ &(\pair{n}{m} + 1 \minus 1)_l\notag\\
		=_0\ &n\period\tag{$*$}
		\intertext{We have $\pair{n}{m} \geq m$ and hence:}
		&q(\widehat{s''k} * k)\notag\\
		=_0\ &q(\langle s''(\hat{u} * \overline{y}(\pair{n}{m})) \rangle * \hat{u} * \overline{y}(\pair{n}{m}))\notag\\
		\overset{(*)}{=}_0\ &q(\hat{n} * \overline{\hat{u} * y}(\pair{n}{m} + 1))\notag\\
		=_0\ &q(\overline{\hat{n} * \hat{u} * y}(\pair{n}{m} + 2))\notag\\
		=_0\ &0\period\tag{$\dagger$}
		\intertext{Furthermore, we have}
		&s'(r(\widehat{s'' k} * k))\notag\\
		=_0\ &s'(r(\langle s''(\hat{u} * \overline{y}(\pair{n}{m})) \rangle * \hat{u} * \overline{y}(\pair{n}{m})))\notag\\
		=_0\ &s'(r(\langle s''(\hat{u} * \overline{y}(\pair{n}{m})) \rangle * \overline{\hat{u} * y}(\pair{n}{m} + 1)))\notag\\
		\overset{(*)}{=}_0\ &s'(\hat{n} * \zerosAfter{\hat{u} * y}{\pair{n}{m} + 1})\notag\\
		=_0\ &s'(\zerosAfter{\hat{n} * \hat{u} * y}{\pair{n}{m} + 2})\period\notag
		\intertext{Now, because $q(\overline{\hat{n} * \hat{u} * y}(\pair{n}{m} + 2)) =_0 0$ holds, we have}
		=_0\ &s'(\hat{n} * \hat{u} * y)\notag\\
		=_0\ &ty\tilde{u}n \comma\notag
		\end{align}
		for which still $ty\tilde{u}n \neq_0 0$ holds, which implies
		\begin{equation*}
		s'(r(\widehat{s''k} * k)) \neq_0 0\period
		\end{equation*}
		This, together with the result ($\dagger$), yields
		\begin{equation*}
		sk \neq_0 0
		\end{equation*}
		and, since $u$ was arbitrary, $\halts{s \cdot y}$.
		
		Now, we show that $\wit(s \cdot y, z)$ holds. Under the assumption $\wit(s \cdot y, w)$, this implies our claim $z =_1 w$ since such a computed value is always unique. We prove that for all natural numbers $u$ and $v$ with $s(\hat{u} * \overline{y} v) \neq_0 0$, we already have $s(\hat{u} * \overline{y} v) =_0 zu + 1$. Combined with $\halts{s \cdot y}$, this will imply $\wit(s \cdot y, z)$. Therefore, assume that $u$ and $v$ are such that $s(\hat{u} * \overline{y} v) \neq_0 0$ holds. By definition of $s$, this yields both $q(\langle s''(\hat{u} * \overline{y} v)\rangle * u * \overline{y} v) =_0 0$ and $s'(r(\langle s''(\hat{u} * \overline{y} v)\rangle * \hat{u} * \overline{y} v)) \neq_0 0$. With $s''(\hat{u} * \overline{y}v) =_0 v_l$, the former one implies both $q(\overline{\hat{v}_l * \hat{u} * y}(v + 2)) =_0 0$ and $s'(\overline{\hat{v}_l * \hat{u} * y}(v + 2)) \neq_0 0$, and therefore $s'(\hat{v}_l * \hat{u} * y) =_0 s'(\zerosAfter{\hat{v}_l * \hat{u} * y}{v+2}) \neq_0 0$. Since $s'(\hat{v}_l * \hat{u} * y) =_0 ty\tilde{u}v_l$ holds, we have $ty\tilde{u}v_l \neq_0 0$, which together with $\con(ty, z)$ yields $s(\hat{u} * \overline{y} v) =_0 ty\tilde{u}v_l =_0 zu + 1$.
	\end{enumerate}
\end{proof}

The following lemma tells us that, under certain circumstances, sequences that have been computed by associates are linear:

\begin{lemma}\label{le:wit}
	Let $t^{0(0)(0)(1)(1)}$ be a closed term such that
	\begin{equation*}
	\fCenter tx^1 y^1 u^0 v^0 \leftrightarrow (x(\hat{u} * \overline{y}v) \neq_0 0) \land \forall w <_0 v (x(\hat{u} * \overline{y}w) =_0 0)
	\end{equation*}
	holds in $\EPAomega$. We define $\Aat(x^1, y^1, u^0, v^0) :\equiv (txyuv =_0 0)$ and can prove the following in $\ELPAomegal$:
	\begin{enumerate}[label=\alph*)]
		\item\label{leit:witHalt2Aqf} $\fCenter {\emL{(\halts{x^1 \cdot y^1})}}^{\bot}, \forall u^0 \exists v^0 \Aat(x, y, u, v)$,
		\item\label{leit:witAqf2Wit} $\fCenter \lp^{\bot}(x), \lp^{\bot}(y), (\existsL V \forall u^0 \Aat(x, y, u, Vu))^{\bot}, \existsL z^1 \emL{\wit}(x \cdot y, z)$,
		\item\label{leit:witIntroC} $\fCenter \lp^{\bot}(x), \lp^{\bot}(y), {\emL{(\halts{x^1 \cdot y^1})}}^{\bot}, \existsL z^1 \emL{\wit}(x \cdot y, z)$.
	\end{enumerate}
\end{lemma}

\begin{proof}\mbox{}
	\begin{enumerate}[label=\alph*)]
		\item Assume $\halts{x^1 \cdot y^1}$. This means that for all $u^0$ there exists some $v^0$ with $x(\hat{u} * \overline{y}v) \neq_0 0$. If we choose $v$ as the least such number, we have $\Aat(x, y, u, v)$. We can therefore prove $\halts{x^1 \cdot y^1} \fCenter \forall u^0 \exists v^0 \Aat(x, y, u, v)$ in $\EPAomega$, and may translate this result to a proof of $\fCenter {\emL{(\halts{x^1 \cdot y^1})}}^{\bot}, \forall u^0 \exists v^0 \Aat(x, y, u, v)$ in $\ELPAomegal$ via Lemma~\ref{lem:embedPeanoIntoLinear}.
		\item We define
		\begin{equation*}
		s := \lambda x^1, y^1, V^1, u^0. x(\hat{u} * \overline{y}(Vu)) \minus 1\period
		\end{equation*}
		Assume that $\Aat(x, y, u, Vu)$ holds for all natural numbers $u$. By the definition of $\wit$ it is clear that $\wit(x \cdot y, sxyV)$ holds. We can therefore prove\
		\begin{equation*}
		\forall u^0 \Aat(x, y, u, Vu) \fCenter \wit(x \cdot y, sxyV)
		\end{equation*}
		in $\EPAomega$. Like above, we may translate this result to a proof of
		\begin{equation*}
		\fCenter (\forall u^0 \Aat(x, y, u, Vu))^{\bot}, \emL{\wit}(x \cdot y, sxyV)
		\end{equation*}	
		in $\ELPAomegal$ via Lemma~\ref{lem:embedPeanoIntoLinear}. We continue with an application of ($\otimes$) to the instance $\vdash \lp^{\bot}(sxyV), \lp(sxyV)$ of (id):
		\begin{equation*}
		\fCenter \lp^{\bot}(sxyV), (\forall u^0 \Aat(x, y, u, Vu))^{\bot}, \lp(sxyV) \otimes \emL{\wit}(x \cdot y, sxyV)\period
		\end{equation*}
		Now, we apply ($\lp$-app) twice and cut the $\lp^{\bot}(s)$ away using (cut) and ($\lp$) since $s$ is a closed term.
		\begin{equation*}
		\fCenter \lp^{\bot}(x), \lp^{\bot}(y), \lp^{\bot}(V), (\forall u^0 \Aat(x, y, u, Vu))^{\bot}, \lp(sxyV) \otimes \emL{\wit}(x \cdot y, sxyV)\period
		\end{equation*}
		Finally, we reach our claim by using ($\Par$), ($\exists$), and ($\forall$):
		\begin{equation*}
		\fCenter \lp^{\bot}(x), \lp^{\bot}(y), (\existsL V \forall u^0 \Aat(x, y, u, Vu))^{\bot}, \existsL z^1 \emL{\wit}(x \cdot y, z)\period
		\end{equation*}
		\item
		We start off by combining two instances of (id) with ($\otimes$):
		\begin{equation*}
			\fCenter (\forall u^0 \Aat(x, y, u, Vu))^{\bot}, \lp^{\bot}(V^1), \lp(V) \otimes \forall u^0 \Aat(x, y, u, Vu)\period
		\end{equation*}
		We introduce a further instance of (id) with ($\otimes$), and apply (d$?$):
		\begin{multline*}
			\fCenter (\forall u^0 \Aat(x, y, u, Vu))^{\bot}, \lp^{\bot}(txy),\\(!(\lp(txy) \lto \lp(V^1)))^{\bot},\lp(V) \otimes \forall u^0 \Aat(x, y, u, Vu)\period
		\end{multline*}
		After the introduction of an existential quantifier and splitting $\lp^{\bot}(txy)$ into three parts $\lp^{\bot}(t)$, $\lp^{\bot}(x)$, and $\lp^{\bot}(y)$ using ($\lp$-app), where we can immediately cut $\lp^{\bot}(t)$ away with ($\lp$) since $t$ is closed, we reach the following sequent:
		\begin{multline*}
		\fCenter (\forall u^0 \Aat(x, y, u, Vu))^{\bot}, \lp^{\bot}(x), \lp^{\bot}(y),\\(!(\lp(txy) \lto \lp(V^1)))^{\bot}, \existsL V^1 \forall u^0 \Aat(x, y, u, Vu)\period
		\end{multline*}
		Now, we apply ($\Par$) and ($\forall$)
		\begin{multline*}
		\fCenter (\exists V (\forall u^0 \Aat(x, y, u, Vu) \otimes !(\lp(txy) \lto \lp(V^1))))^{\bot},\\\lp^{\bot}(x), \lp^{\bot}(y),\existsL V^1 \forall u^0 \Aat(x, y, u, Vu)
		\end{multline*}
		and immediately cut with ($\lAC$) where $\alpha :\equiv txy$:
		\begin{equation*}
		\fCenter (\forall u^0 \exists v^0 \Aat(x, y, u, v))^{\bot} ,\lp^{\bot}(x), \lp^{\bot}(y),\existsL V^1 \forall u^0 \Aat(x, y, u, Vu)\period
		\end{equation*}
		Finally, by cutting with a) and b), and contracting the resulting double appearances of $\lp^{\bot}(x)$ and $\lp^{\bot}(y)$ using ($\lp$-con), we reach our claim:
		\begin{equation*}
		\fCenter \lp^{\bot}(x), \lp^{\bot}(y), {\emL{(\halts{x^1 \cdot y^1})}}^{\bot}, \existsL z^1 \emL{\wit}(x \cdot y, z)\period
		\end{equation*}
	\end{enumerate}
\end{proof}

\section{Dialectica Interpretation of Linear Logic with Linear Predicate}

For the functional interpretation we choose a compact notation inspired by \cite{olivaClassical}.

\begin{definition}[Functional interpretation]
	For every formula $A$ of $\ELPAomegaeq$ we define a formula $\Real{A}{\tupl{x}}{\tupl{y}}$ of $\ELPAomegaeq$ together with tuples of fresh variables $\tupl{x}$ and $\tupl{y}$, inductively. We omit $\tupl{x}$, resp. $\tupl{y}$, in the notation if the tuple is empty.
	\begin{alignat}{3}
	&\Real{A}{}{} &&:\equiv A \text{ for } A \in \{1, 0, \top, \bot, s =_0 t, s \doteq_0 t, \place(t)\}\period\notag\\
	\intertext{For the interpretation of the linear predicate, we introduce two possible definitions:}
	&\Real{\lp_{\tau}(t)}{x^0}{} &&:\equiv \lp_{\tau}(t) \otimes \place(\sg(x))\comma\tag{a}\\
	&\Real{\lp_{\tau}(t)}{x^{\comp{\tau}}}{} &&:\equiv \emL{\con}_{\tau}(x, t)\period\tag{b}\\
	\intertext{We assume that $\Real{A}{\tupl{x}}{\tupl{y}}$ and $\Real{B}{\tupl{u}}{\tupl{v}}$ have already been defined:}
	&\Real{A^{\bot}}{\tupl{u}}{\tupl{v}} &&:\equiv (\Real{A}{\tupl{v}}{\tupl{u}})^{\bot} \text{ for unnegated atomic formulas $A$}\comma\notag\\
	&\Real{A \oplus B}{\tupl{x}, \tupl{u}, k^0}{\tupl{y}, \tupl{v}} &&:\equiv (!k\doteq_0 0 \otimes \Real{A}{\tupl{x}}{\tupl{y}}) \oplus (!k \not\doteq_0 0 \otimes \Real{B}{\tupl{u}}{\tupl{v}})\comma\notag\\
	&\Real{A \with B}{\tupl{x}, \tupl{u}}{\tupl{y}, \tupl{v}, k^0} &&:\equiv (!k\doteq_0 0 \lto \Real{A}{\tupl{x}}{\tupl{y}}) \with (!k \not\doteq_0 0 \lto \Real{B}{\tupl{u}}{\tupl{v}})\comma\notag\\
	&\Real{A \Par B}{\tupl{f}, \tupl{g}}{\tupl{x}, \tupl{u}} &&:\equiv \Real{A}{\tupl{f}\tupl{u}}{\tupl{x}} \Par \Real{B}{\tupl{g}\tupl{x}}{\tupl{u}}\comma\notag\\
	&\Real{A \otimes B}{\tupl{x}, \tupl{u}}{\tupl{f}, \tupl{g}} &&:\equiv \Real{A}{\tupl{x}}{\tupl{f}\tupl{u}} \otimes \Real{B}{\tupl{u}}{\tupl{g}\tupl{x}}\comma\notag\\
	&\Real{\exists z A}{\tupl{x}}{\tupl{y}} &&:\equiv \exists z \Real{A}{\tupl{x}}{\tupl{y}}\comma\notag\\
	&\Real{\forall z A}{\tupl{x}}{\tupl{y}} &&:\equiv \forall z \Real{A}{\tupl{x}}{\tupl{y}}\comma\notag\\
	&\Real{?A}{}{\tupl{y}} &&:\equiv ?\exists \tupl{x} \Real{A}{\tupl{x}}{\tupl{y}}\comma\notag\\
	&\Real{!A}{\tupl{x}}{} &&:\equiv !\forall \tupl{y} \Real{A}{\tupl{x}}{\tupl{y}}\period\notag
	\end{alignat}
\end{definition}

\begin{lemma}\label{Lemma:InvolutionInterpretation}
	For all formulas $A$ of $\ELPAomegal$, we have $\Real{A^{\bot}}{\tupl{y}}{\tupl{x}} \equiv {\Real{A}{\tupl{x}}{\tupl{y}}}^{\bot}$.
\end{lemma}

\begin{proof}
	The proof proceeds by induction on the structure of $A$. For atomic formulas the result is clear by definition. Finally, we use that the interpretations of non-atomic formulas are well-behaved with respect to involution. For example, under the assumption that the induction hypothesis holds for $A$ and $B$, we can show the following for the formula $A \with B$:
	\begin{align*}
		\Real{(A \with B)^{\bot}}{\tupl{x}, \tupl{u}, k^0}{\tupl{y}, \tupl{v}} &\equiv \Real{(A^{\bot} \oplus B^{\bot})}{\tupl{x}, \tupl{u}, k^0}{\tupl{y}, \tupl{v}}\\
		&\equiv (!k\doteq_0 0 \otimes \Real{A^{\bot}}{\tupl{x}}{\tupl{y}}) \oplus (!k \not\doteq_0 0 \otimes \Real{B^{\bot}}{\tupl{u}}{\tupl{v}})\\
		&\equiv (!k\doteq_0 0 \otimes (\Real{A}{\tupl{y}}{\tupl{x}})^{\bot}) \oplus (!k \not\doteq_0 0 \otimes (\Real{B}{\tupl{v}}{\tupl{u}})^{\bot})\\
		&\equiv (!k\doteq_0 0 \lto \Real{A}{\tupl{y}}{\tupl{x}})^{\bot} \oplus (!k \not\doteq_0 0 \lto \Real{B}{\tupl{v}}{\tupl{u}})^{\bot}\\
		&\equiv ((!k\doteq_0 0 \lto \Real{A}{\tupl{y}}{\tupl{x}}) \with (!k \not\doteq_0 0 \lto \Real{B}{\tupl{v}}{\tupl{u}}))^{\bot}\\
		&\equiv (\Real{A \with B}{\tupl{y}, \tupl{v}}{\tupl{x}, \tupl{u}, k^0})^{\bot}\period
	\end{align*}
\end{proof}

\begin{theorem}\label{Theorem:Dialectica}
	Let $A_1$, \dots, $A_n$ be formulas of $\ELPAomegaeq$, and $\Gamma$ a set of formulas in $\EPAomega$, and assume that $\ELPAomegaeq + \emL{\Gamma}$ (or $\EAPAomegaeq + \emL{\Gamma}$) proves
	\begin{equation*}
	\vdash A_1, \dots, A_n\period
	\end{equation*}
	If the functional interpretation of $\lp_{\tau}(t)$ is defined as
	\begin{enumerate}[label=\alph*)]
		\item $\Real{\lp_{\tau}(t)}{x^0}{} :\equiv \lp_{\tau}(t) \otimes \place(\sg(x))$, or as
		\item $\Real{\lp_{\tau}(t)}{x^{\comp{\tau}}}{} :\equiv \emL{\con}_{\tau}(x, t)$,
	\end{enumerate}
	then
	$\ELPAomegaeq + \emL{\Gamma}$ (or $\EAPAomegaeq + \emL{\Gamma}$) proves
	\begin{equation*}
	\fCenter \Real{A_1}{\tupl{a}_1}{\tupl{x}_1}, \dots, \Real{A_n}{\tupl{a}_n}{\tupl{x}_n}
	\end{equation*}
	for tuples of terms $\tupl{a}_1$, \dots, $\tupl{a}_n$ where the free variables of each $\tupl{a}_i$ are among those in the sequence of terms $\tupl{x}_1, \dots, \tupl{x}_{i-1}, \tupl{x}_{i+1}, \dots, \tupl{x}_n$. In particular, the variables $\tupl{x}_i$ are not free in $\tupl{a}_i$.
\end{theorem}

\begin{remark}
	The first interpretation a) for $\lp$ will be used to apply tags to the linear predicates for the theorem of Section \ref{sec:Phase}. Since we are only interested in two different kinds of tags, we use $\sg$ in order to limit the information given to the tagging predicate $\place$ to the values $0$ and $1$.
	
	The second interpretation b) for $\lp$ will be used to extract the computational information that lies hidden within the proof of the implication from one Weihrauch problem to the other. This interpretation is inspired by that for the standard predicate in \cite{nonstandard}. While their interpretation of $\Real{\lp_{\tau}(t)}{x}{}$ would be $x = t$, ours has to be more indirect with $\con_{\tau}(x, t)$, i.e.~$x$ is a term that merely \emph{constructs} the term $t$, because we need to interpret the axiom $(\lAC)$, which is crucial for the characterization of \emph{computable} Weihrauch reducibility as could be seen during its application in the proof of Lemma \ref{le:wit} c).
\end{remark}

\begin{proof}
	We prove the theorem by induction on the length of derivations in $\ELPAomegaeq$.
	\begin{itemize}
		\item Nonlinear axioms including those from $\emL{\Gamma}$:
		
		Nonlinear formulas $\Anl$ have the property that their functional interpretation is identical to themselves, i.e.~$\Real{\Anl}{}{}\equiv \Anl$ holds. Therefore, the proofs of their functional interpretations are trivial.
		The rules (tag-con), (tag-app), (dot-eq$_1$), (dot-eq$_2$), ($\top$), and ($\bot$) fall into a similar category.
		
		\item Identity and structure:
		\begin{center}
			\begin{bprooftree}
				\AxiomC{}
				\RightLabel{(id)}
				\UnaryInf$\fCenter \Real{A}{\tupl{u}}{\tupl{v}}, (\Real{A}{\tupl{u}}{\tupl{v}})^{\bot}$
				\UnaryInf$\fCenter \Real{A}{\tupl{u}}{\tupl{v}}, \Real{A^{\bot}}{\tupl{v}}{\tupl{u}}$
			\end{bprooftree}
			\hskip 1.5em
			\begin{bprooftree}
				\Axiom$\fCenter \Real{G_1}{\tupl{g_1}}{\tupl{u_1}}, \dots, \Real{G_n}{\tupl{g_n}}{\tupl{u_n}}$
				\RightLabel{(per)}
				\UnaryInf$\fCenter \Real{G_{\pi1}}{\tupl{g_{\pi1}}}{\tupl{u_{\pi1}}}, \dots, \Real{G_{\pi n}}{\tupl{g_{\pi n}}}{\tupl{u_{\pi n}}}$
			\end{bprooftree}
		\end{center}
		\begin{prooftree}
			\Axiom$\fCenter \Real{G_1}{\tupl{g_1}}{\tupl{u_1}}, \dots,\Real{G_n}{\tupl{g_n}}{\tupl{u_n}}, \Real{A}{\tupl{a}}{\tupl{v}}$
			\RightLabel{$([\nicefrac{\tupl{b}}{\tupl{v}}])$}
			\UnaryInf$\fCenter \Real{G_1}{\tupl{g'_1}}{\tupl{u_1}}, \dots,\Real{G_n}{\tupl{g'_n}}{\tupl{u_n}}, \Real{A}{\tupl{a}}{\tupl{b}}$
			\Axiom$\fCenter \Real{A^{\bot}}{\tupl{b}}{\tupl{w}}, \Real{D_1}{\tupl{d_1}}{\tupl{z_1}}, \dots, \Real{D_m}{\tupl{d_m}}{\tupl{z_m}}$
			\UnaryInf$\fCenter (\Real{A}{\tupl{w}}{\tupl{b}})^{\bot}, \Real{D_1}{\tupl{d_1}}{\tupl{z_1}}, \dots, \Real{D_m}{\tupl{d_m}}{\tupl{z_m}}$
			\RightLabel{($[\nicefrac{\tupl{a}}{\tupl{w}}]$)}
			\UnaryInf$\fCenter (\Real{A}{\tupl{a}}{\tupl{b}})^{\bot}, \Real{D_1}{\tupl{d'_1}}{\tupl{z_1}}, \dots, \Real{D_m}{\tupl{d'_m}}{\tupl{z_m}}$
			\RightLabel{(cut)}
			\BinaryInf$\fCenter \Real{G_1}{\tupl{g'_1}}{\tupl{u_1}}, \dots,\Real{G_n}{\tupl{g'_n}}{\tupl{u_n}}, \Real{D_1}{\tupl{d'_1}}{\tupl{z_1}}, \dots, \Real{D_m}{\tupl{d'_m}}{\tupl{z_m}}$
		\end{prooftree}
		where  $\pi$ is an arbitrary permutation.
		\item Multiplicatives:
		\begin{prooftree}
			\Axiom$\fCenter \Real{G_1}{\tupl{g_1}}{\tupl{u_1}}, \dots, \Real{G_n}{\tupl{g_n}}{\tupl{u_n}}, \Real{A}{\tupl{a}}{\tupl{w}}$
			\RightLabel{$([\nicefrac{\tupl{x}\tupl{b}}{\tupl{w}}])$}
			\UnaryInf$\fCenter \Real{G_1}{\tupl{g'_1}}{\tupl{u_1}}, \dots, \Real{G_n}{\tupl{g'_n}}{\tupl{u_n}}, \Real{A}{\tupl{a}}{\tupl{x}\tupl{b}}$
			\Axiom$\fCenter \Real{D_1}{\tupl{d_1}}{\tupl{v_1}}, \dots, \Real{D_n}{\tupl{d_n}}{\tupl{v_n}}, \Real{B}{\tupl{b}}{\tupl{z}}$
			\RightLabel{($[\nicefrac{\tupl{y}\tupl{a}}{\tupl{z}}]$)}
			\UnaryInf$\fCenter \Real{D_1}{\tupl{d'_1}}{\tupl{v_1}}, \dots, \Real{D_n}{\tupl{d'_n}}{\tupl{v_n}}, \Real{B}{\tupl{b}}{\tupl{y}\tupl{a}}$
			\RightLabel{($\otimes$)}
			\BinaryInf$\fCenter \Real{G_1}{\tupl{g'_1}}{\tupl{u_1}}, \dots, \Real{G_n}{\tupl{g'_n}}{\tupl{u_n}}, \Real{D_1}{\tupl{d'_1}}{\tupl{v_1}}, \dots, \Real{D_n}{\tupl{d'_n}}{\tupl{v_n}},  \Real{A}{\tupl{a}}{\tupl{x}\tupl{b}} \otimes \Real{B}{\tupl{b}}{\tupl{y}\tupl{a}}$
			\UnaryInf$\fCenter \Real{G_1}{\tupl{g'_1}}{\tupl{u_1}}, \dots, \Real{G_n}{\tupl{g'_n}}{\tupl{u_n}}, \Real{D_1}{\tupl{d'_1}}{\tupl{v_1}}, \dots, \Real{D_n}{\tupl{d'_n}}{\tupl{v_n}},  \Real{A \otimes B}{\tupl{a}, \tupl{b}}{\tupl{x}, \tupl{y}}$
		\end{prooftree}
		\begin{center}
			\begin{bprooftree}
				\Axiom$\fCenter \Gamma, \Real{A}{\tupl{a}}{\tupl{u}}, \Real{B}{\tupl{b}}{\tupl{v}}$
				\RightLabel{($*$)}
				\UnaryInf$\fCenter \Gamma, \Real{A}{(\lambda \tupl{v}.\tupl{a})\tupl{v}}{\tupl{u}}, \Real{B}{(\lambda \tupl{u}.\tupl{b})\tupl{u}}{\tupl{v}}$
				\UnaryInf$\fCenter \Gamma, \Real{A \Par B}{(\lambda \tupl{v}.\tupl{a}), (\lambda \tupl{u}.\tupl{b})}{\tupl{u}, \tupl{v}}$
			\end{bprooftree}
		\end{center}
		where we use of Lemma \ref{lem:embedPeanoIntoLinear}, (dot-eq$_1$), and (dot-sub) on the facts $\tupl{a}_i =_{\tau} (\lambda \tupl{v}. \tupl{a})_i \tupl{v}$, in order to cut from $\Real{A}{\tupl{a}}{\tupl{u}}$ to $\Real{A}{(\lambda \tupl{v}.\tupl{a})\tupl{v}}{\tupl{u}}$ in ($*$). We apply the analogous argument to $\Real{B}{\tupl{b}}{\tupl{v}}$.
		\item Additives:
		
		We break the proof for the interpretation of the ($\with$) rule into three parts:
		\begin{prooftree}
			\Axiom$\fCenter \Real{G_1}{\tupl{g_1}}{\tupl{u_1}}, \dots, \Real{G_n}{\tupl{g_n}}{\tupl{u_n}}, \Real{A}{\tupl{a}}{\tupl{v}}$
			\RightLabel{($*$)}
			\UnaryInf$\fCenter \Real{G_1}{\tupl{t_1}\sg(0)\tupl{g_1}\tupl{g'_1}}{\tupl{u_1}}, \dots, \Real{G_n}{\tupl{t_n}\sg(0)\tupl{g_n}\tupl{g'_n}}{\tupl{u_n}}, \Real{A}{\tupl{a}}{\tupl{v}}$
			\RightLabel{($\dagger$)}
			\UnaryInf$\fCenter (! k \doteq_0 0)^{\bot}, \Real{G_1}{\tupl{t_1}\sg(k)\tupl{g_1}\tupl{g'_1}}{\tupl{u_1}}, \dots, \Real{G_n}{\tupl{t_n}\sg(k)\tupl{g_n}\tupl{g'_n}}{\tupl{u_n}}, \Real{A}{\tupl{a}}{\tupl{v}}$
			\RightLabel{($\Par$)}
			\UnaryInf$\fCenter \Real{G_1}{\tupl{t_1}\sg(k)\tupl{g_1}\tupl{g'_1}}{\tupl{u_1}}, \dots, \Real{G_n}{\tupl{t_n}\sg(k)\tupl{g_n}\tupl{g'_n}}{\tupl{u_n}}, ! k \doteq_0 0 \lto \Real{A}{\tupl{a}}{\tupl{v}}$
		\end{prooftree}
		where the tuples $\tupl{t_i}$ are defined such that $\tupl{t_i}\sg(0) \tupl{g_i} \tupl{g'_i} = \tupl{g_i}$ and $\tupl{t_i} 1 \tupl{g_i} \tupl{g'_i} = \tupl{g'_i}$ are provable $\EPAomega$. We transport this fact to $\ELPAomegaeq$ using Lemma \ref{lem:embedPeanoIntoLinear} and (dot-eq$_1$). We apply (dot-sub) to these equalities and cut with them in ($*$). In ($\dagger$), we cut with further applications of the (dot-sub) rule. Here, we also combine multiple occurrences of $(!k \doteq_0 0)^{\bot}$ into a single one using (c?).
		\begin{prooftree}
			\Axiom$\fCenter \Real{G_1}{\tupl{g'_1}}{\tupl{u_1}}, \dots, \Real{G_n}{\tupl{g'_n}}{\tupl{u_n}}, \Real{B}{\tupl{b}}{\tupl{w}}$
			\RightLabel{($*$)}
			\UnaryInf$\fCenter \Real{G_1}{\tupl{t_1}1\tupl{g_1}\tupl{g'_1}}{\tupl{u_1}}, \dots, \Real{G_n}{\tupl{t_n}1\tupl{g_n}\tupl{g'_n}}{\tupl{u_n}}, \Real{B}{\tupl{b}}{\tupl{w}}$
			\RightLabel{($\dagger$)}
			\UnaryInf$\fCenter (! \sg(k) \doteq_0 1)^{\bot}, \Real{G_1}{\tupl{t_1}\sg(k)\tupl{g_1}\tupl{g'_1}}{\tupl{u_1}}, \dots, \Real{G_n}{\tupl{t_n}\sg(k)\tupl{g_n}\tupl{g'_n}}{\tupl{u_n}}, \Real{B}{\tupl{b}}{\tupl{w}}$
			\RightLabel{(dot-eq$_2$*)}
			\UnaryInf$\fCenter (! k \not\doteq_0 0)^{\bot}, \Real{G_1}{\tupl{t_1}\sg(k)\tupl{g_1}\tupl{g'_1}}{\tupl{u_1}}, \dots, \Real{G_n}{\tupl{t_n}\sg(k)\tupl{g_n}\tupl{g'_n}}{\tupl{u_n}}, \Real{B}{\tupl{b}}{\tupl{w}}$
			\RightLabel{($\Par$)}
			\UnaryInf$\fCenter \Real{G_1}{\tupl{t_1}\sg(k)\tupl{g_1}\tupl{g'_1}}{\tupl{u_1}}, \dots, \Real{G_n}{\tupl{t_n}\sg(k)\tupl{g_n}\tupl{g'_n}}{\tupl{u_n}}, ! k \not\doteq_0 0 \lto \Real{B}{\tupl{b}}{\tupl{w}}$
		\end{prooftree}
		where both ($*$) and ($\dagger$) are similar to before. Furthermore, in (dot-eq$_2$*) we cut with (dot-eq$_2$).
		Finally, we combine both previous results using the ($\with$) rule in order to prove the following:
		\begin{prooftree}
			\Axiom$\fCenter \Real{G_1}{\tupl{t_1}\sg(k)\tupl{g_1}\tupl{g'_1}}{\tupl{u_1}}, \dots, \Real{G_n}{\tupl{t_n}\sg(k)\tupl{g_n}\tupl{g'_n}}{\tupl{u_n}}, (! k \doteq_0 0 \lto \Real{A}{\tupl{a}}{\tupl{v}}) \with (! k \not\doteq_0 0 \lto \Real{B}{\tupl{b}}{\tupl{w}})$
			\UnaryInf$\fCenter \Real{G_1}{\tupl{t_1}\sg(k)\tupl{g_1}\tupl{g'_1}}{\tupl{u_1}}, \dots, \Real{G_n}{\tupl{t_n}\sg(k)\tupl{g_n}\tupl{g'_n}}{\tupl{u_n}}, \Real{A \with B}{\tupl{a}, \tupl{b}}{\tupl{v}, \tupl{w}, k}$
		\end{prooftree}
		which concludes the proof. We continue with the interpretations of ($\oplus_1$) and ($\oplus_2$):
		\begin{center}
			\begin{bprooftree}
				\AxiomC{}
				\RightLabel{(refl)}
				\UnaryInf$\fCenter 0 =_0 0$
				\RightLabel{(dot-eq$_1$)}
				\UnaryInf$\fCenter ! 0 \doteq_0 0$
				\Axiom$\fCenter \Gamma, \Real{A}{\tupl{a}}{\tupl{u}}$
				\RightLabel{($\otimes$)}
				\BinaryInf$\fCenter \Gamma, ! 0 \doteq_0 0 \otimes \Real{A}{\tupl{a}}{\tupl{u}}$
				\RightLabel{($\oplus_1$)}
				\UnaryInf$\fCenter \Gamma, (! 0 \doteq_0 0 \otimes \Real{A}{\tupl{a}}{\tupl{u}}) \oplus (! 0 \not\doteq_0 0 \otimes \Real{B}{\tupl{b}}{\tupl{v}})$
				\UnaryInf$\fCenter \Gamma, \Real{A \oplus B}{\tupl{a}, \tupl{b}, 0}{\tupl{u}, \tupl{v}}$
			\end{bprooftree}
			\vskip 1em
			\begin{bprooftree}
				\AxiomC{}
				\RightLabel{(dot-succ)}
				\UnaryInf$\fCenter !1 \not\doteq_0 0$
				\Axiom$\fCenter \Gamma, \Real{B}{\tupl{b}}{\tupl{v}}$
				\RightLabel{($\otimes$)}
				\BinaryInf$\fCenter \Gamma, !1 \not\doteq_0 0 \otimes \Real{B}{\tupl{b}}{\tupl{v}}$
				\RightLabel{($\oplus_2$)}
				\UnaryInf$\fCenter \Gamma, (!1 \doteq_0 0 \otimes \Real{A}{\tupl{a}}{\tupl{u}}) \oplus (!1 \not\doteq_0 0 \otimes \Real{B}{\tupl{b}}{\tupl{v}})$
				\UnaryInf$\fCenter \Gamma, \Real{A \oplus B}{\tupl{a}, \tupl{b}, 1}{\tupl{u}, \tupl{v}}$
			\end{bprooftree}
		\end{center}
		\item Modalities:
		\begin{center}
			\begin{bprooftree}
				\Axiom$\fCenter \Real{?G_1}{}{\tupl{y}_1}, \dots, \Real{?G_n}{}{\tupl{y}_n}, \Real{A}{\tupl{a}}{\tupl{u}}$
				\UnaryInf$\fCenter ?\exists \tupl{x_1}\Real{G_1}{\tupl{x_1}}{\tupl{y}_1}, \dots, ?\exists\tupl{x_n}\Real{G_n}{\tupl{x_n}}{\tupl{y}_n}, \Real{A}{\tupl{a}}{\tupl{u}}$
				\RightLabel{($\forall$)}
				\UnaryInf$\fCenter ?\exists \tupl{x_1}\Real{G_1}{\tupl{x_1}}{\tupl{y}_1}, \dots, ?\exists\tupl{x_n}\Real{G_n}{\tupl{x_n}}{\tupl{y}_n}, \forall\tupl{u}\Real{A}{\tupl{a}}{\tupl{u}}$
				\RightLabel{($!$)}
				\UnaryInf$\fCenter ?\exists \tupl{x_1}\Real{G_1}{\tupl{x_1}}{\tupl{y}_1}, \dots, ?\exists\tupl{x_n}\Real{G_n}{\tupl{x_n}}{\tupl{y}_n}, !\forall\tupl{u}\Real{A}{\tupl{a}}{\tupl{u}}$
				\UnaryInf$\fCenter \Real{?G_1}{}{\tupl{y}_1}, \dots, \Real{?G_n}{}{\tupl{y}_n},\Real{!A}{\tupl{a}}{}$
			\end{bprooftree}
			\hskip 1.5em
			\begin{bprooftree}
				\Axiom$\fCenter \Gamma, \Real{A}{\tupl{a}}{\tupl{v}}$
				\RightLabel{($\exists$)}
				\UnaryInf$\fCenter \Gamma, \exists\tupl{u}\Real{A}{\tupl{u}}{\tupl{v}}$
				\RightLabel{(d$?$)}
				\UnaryInf$\fCenter \Gamma, ?\exists\tupl{u}\Real{A}{\tupl{u}}{\tupl{v}}$
				\UnaryInf$\fCenter \Gamma, \Real{?A}{}{\tupl{v}}$
			\end{bprooftree}
			\vskip 1em
			\begin{bprooftree}
				\Axiom$\fCenter \Gamma$
				\RightLabel{(w$?$)}
				\UnaryInf$\fCenter \Gamma, ?\exists\tupl{u}\Real{A}{\tupl{u}}{\tupl{v}}$
				\UnaryInf$\fCenter \Gamma, \Real{?A}{}{\tupl{v}}$
			\end{bprooftree}
			\hskip 1.5em
			\begin{bprooftree}
				\Axiom$\fCenter \Real{G_1}{\tupl{g_1}}{\tupl{u_1}}, \dots, \Real{G_n}{\tupl{g_n}}{\tupl{u_n}}, \Real{?A}{}{\tupl{y}}, \Real{?A}{}{\tupl{v}}$
				\RightLabel{($[\nicefrac{\tupl{y}}{\tupl{v}}]$)}
				\UnaryInf$\fCenter \Real{G_1}{\tupl{g'_1}}{\tupl{u_1}}, \dots, \Real{G_n}{\tupl{g'_n}}{\tupl{u_n}}, \Real{?A}{}{\tupl{y}}, \Real{?A}{}{\tupl{y}}$
				\RightLabel{($*$)}
				\UnaryInf$\fCenter \Real{G_1}{\tupl{g'_1}}{\tupl{u_1}}, \dots, \Real{G_n}{\tupl{g'_n}}{\tupl{u_n}}, ?\exists\tupl{x}\Real{A}{\tupl{x}}{\tupl{y}}, ?\exists\tupl{x}\Real{A}{\tupl{x}}{\tupl{y}}$
				\RightLabel{(c$?$)}
				\UnaryInf$\fCenter \Real{G_1}{\tupl{g'_1}}{\tupl{u_1}}, \dots, \Real{G_n}{\tupl{g'_n}}{\tupl{u_n}}, ?\exists\tupl{x}\Real{A}{\tupl{x}}{\tupl{y}}$
				\UnaryInf$\fCenter \Real{G_1}{\tupl{g'_1}}{\tupl{u_1}}, \dots, \Real{G_n}{\tupl{g'_n}}{\tupl{u_n}}, \Real{?A}{}{\tupl{y}}$
			\end{bprooftree}
		\end{center}
		where in ($*$) we may use (id), ($\exists$), ($\forall$), and (cut) in order to make sure that the existential quantifiers in the next line are identical.
		\item Quantifiers:
		\begin{center}
			\begin{bprooftree}
				\Axiom$\fCenter \Gamma, \Real{A}{\tupl{a}}{\tupl{u}}$
				\RightLabel{($\forall$)}
				\UnaryInf$\fCenter \Gamma, \forall x \Real{A}{\tupl{a}}{\tupl{u}}$
				\UnaryInf$\fCenter \Gamma, \Real{\forall x A}{\tupl{a}}{\tupl{u}}$
			\end{bprooftree}
			\hskip 1.5em
			\begin{bprooftree}
				\Axiom$\fCenter \Gamma, \Real{A[\nicefrac{t}{x}]}{\tupl{a}}{\tupl{u}}$
				\UnaryInf$\fCenter \Gamma, \Real{A}{\tupl{a}}{\tupl{u}}[\nicefrac{t}{x}]$
				\RightLabel{($\exists$)}
				\UnaryInf$\fCenter \Gamma, \exists x \Real{A}{\tupl{a}}{\tupl{u}}$
				\UnaryInf$\fCenter \Gamma, \Real{\exists x A}{\tupl{a}}{\tupl{u}}$
			\end{bprooftree}
		\end{center}
		\item Linear predicate:
		\begin{enumerate}[label=\alph*)]
		\item\mbox{}\\
		\begin{center}
			\begin{bprooftree}
				\AxiomC{}
				\RightLabel{($\lp$)}
				\UnaryInf$\fCenter \lp_{\tau}(t)$
				\AxiomC{}
				\RightLabel{(tag-0)}
				\UnaryInf$\fCenter \place(\sg(0))$
				\RightLabel{($\otimes$)}
				\BinaryInf$\fCenter \lp_{\tau}(t) \otimes \place(\sg(0))$
				\UnaryInf$\fCenter \Real{\lp_{\tau}(t)}{0}{}$
			\end{bprooftree}
			\vskip 1.5em
			\begin{bprooftree}
				\Axiom$\fCenter \Real{G_1}{\tupl{g_1}}{\tupl{u_1}}, \dots, \Real{G_n}{\tupl{g_n}}{\tupl{u_n}}, \Real{\lp^{\bot}_{\tau}(t)}{}{y}, \Real{\lp^{\bot}_{\tau}(t)}{}{v}$
				\UnaryInf$\fCenter \Real{G_1}{\tupl{g_1}}{\tupl{u_1}}, \dots, \Real{G_n}{\tupl{g_n}}{\tupl{u_n}}, (\lp_{\tau}(t) \otimes \place(\sg(y)))^{\bot}, (\lp_{\tau}(t) \otimes \place(\sg(v)))^{\bot}$
				\RightLabel{($*$)}
				\UnaryInf$\fCenter \Real{G_1}{\tupl{g_1}}{\tupl{u_1}}, \dots, \Real{G_n}{\tupl{g_n}}{\tupl{u_n}}, \lp^{\bot}_{\tau}(t), \place^{\bot}(\sg(y)), \lp^{\bot}_{\tau}(t), \place^{\bot}(\sg(v))$
				\RightLabel{($\lp$-con)}
				\UnaryInf$\fCenter \Real{G_1}{\tupl{g_1}}{\tupl{u_1}}, \dots, \Real{G_n}{\tupl{g_n}}{\tupl{u_n}}, \lp^{\bot}_{\tau}(t), \place^{\bot}(\sg(y)), \place^{\bot}(\sg(v))$
				\RightLabel{([$\nicefrac{y}{v}$])}
				\UnaryInf$\fCenter \Real{G_1}{\tupl{g'_1}}{\tupl{u_1}}, \dots, \Real{G_n}{\tupl{g'_n}}{\tupl{u_n}}, \lp^{\bot}_{\tau}(t), \place^{\bot}(\sg(y)), \place^{\bot}(\sg(y))$
				\RightLabel{(tag-con)}
				\UnaryInf$\fCenter \Real{G_1}{\tupl{g'_1}}{\tupl{u_1}}, \dots, \Real{G_n}{\tupl{g'_n}}{\tupl{u_n}}, \lp^{\bot}_{\tau}(t), \place^{\bot}(\sg(y))$
				\RightLabel{($\Par$)}
				\UnaryInf$\fCenter \Real{G_1}{\tupl{g'_1}}{\tupl{u_1}}, \dots, \Real{G_n}{\tupl{g'_n}}{\tupl{u_n}}, (\lp_{\tau}(t) \otimes \place(\sg(y)))^{\bot}$
				\UnaryInf$\fCenter \Real{G_1}{\tupl{g'_1}}{\tupl{u_1}}, \dots, \Real{G_n}{\tupl{g'_n}}{\tupl{u_n}}, \Real{\lp^{\bot}_{\tau}(t)}{}{y}$
			\end{bprooftree}
			\vskip 1.5em
			\begin{bprooftree}
				\AxiomC{}
				\RightLabel{(tag-app)}
				\UnaryInf$\fCenter \place^{\bot}(\sg(x)), \place^{\bot}(\sg(y)), \place(\sg(x+y))$
				\RightLabel{($\otimes$*)}
				\UnaryInf$\fCenter \lp_{\tau \rho}^{\bot}(t), \place^{\bot}(\sg(x)), \lp_{\rho}^{\bot}(r), \place^{\bot}(\sg(y)), \lp_{\tau}(tr) \otimes \place(\sg(x+y))$
				\RightLabel{($\Par$)}
				\UnaryInf$\fCenter \lp_{\tau \rho}^{\bot}(t) \Par \place^{\bot}(\sg(x)), \lp_{\rho}^{\bot}(r) \Par \place^{\bot}(\sg(y)), \lp_{\tau}(tr) \otimes \place(\sg(x+y))$
				\UnaryInf$\fCenter \Real{\lp^{\bot}_{\tau\rho}(t)}{}{x}, \Real{\lp^{\bot}_{\rho}(r)}{}{y}, \Real{\lp_{\tau}(tr)}{x + y}{}$
			\end{bprooftree}
		\end{center}	
		where in ($*$) we cut twice with the easily provable $\vdash \lp_{\tau}^{\bot}(t), \place^{\bot}(\sg(z)), \lp_{\tau}(t) \otimes \place(\sg(z))$ for $z := y$ and $z := v$, and in ($\otimes$*) we apply ($\otimes$) with the instance $\vdash \lp^{\bot}_{\tau\rho}(t), \lp^{\bot}_{\rho}(r), \lp_{\tau}(tr)$ of ($\lp$-app).
		\item\mbox{}\\
		\begin{center}
			\begin{bprooftree}
				\AxiomC{}
				\RightLabel{($*$)}
				\UnaryInf$\fCenter \emL{\con}_{\tau}(s^{\comp{\tau}}, t^{\tau})$
				\UnaryInf$\fCenter \Real{\lp_{\tau}(t)}{s}{}$
			\end{bprooftree}
			\vskip 1.5em
			\begin{bprooftree}
				\Axiom$\fCenter \Real{G_1}{\tupl{g_1}}{\tupl{u_1}}, \dots, \Real{G_n}{\tupl{g_n}}{\tupl{u_n}}, \Real{\lp^{\bot}_{\tau}(t)}{}{y}, \Real{\lp^{\bot}_{\tau}(t)}{}{v}$
				\UnaryInf$\fCenter \Real{G_1}{\tupl{g_1}}{\tupl{u_1}}, \dots, \Real{G_n}{\tupl{g_n}}{\tupl{u_n}}, (\emL{\con}(y, t))^{\bot}, (\emL{\con}(v, t))^{\bot}$
				\RightLabel{([$\nicefrac{y}{v}$])}
				\UnaryInf$\fCenter \Real{G_1}{\tupl{g'_1}}{\tupl{u_1}}, \dots, \Real{G_n}{\tupl{g'_n}}{\tupl{u_n}}, (\emL{\con}(y, t))^{\bot}, (\emL{\con}(y, t))^{\bot}$
				\RightLabel{($\dagger$)}
				\UnaryInf$\fCenter \Real{G_1}{\tupl{g'_1}}{\tupl{u_1}}, \dots, \Real{G_n}{\tupl{g'_n}}{\tupl{u_n}}, (\emL{\con}(y, t))^{\bot}$
				\UnaryInf$\fCenter \Real{G_1}{\tupl{g'_1}}{\tupl{u_1}}, \dots, \Real{G_n}{\tupl{g'_n}}{\tupl{u_n}}, \Real{\lp^{\bot}_{\tau}(t)}{}{y}$
			\end{bprooftree}
			\vskip 1.5em
			\begin{bprooftree}
				\AxiomC{}
				\RightLabel{($+$)}
				\UnaryInf$\fCenter (\emL{\con_{\tau\rho}}(y, t))^{\bot}, (\emL{\con_{\rho}}(v, r))^{\bot}, \emL{\con_{\tau}}(yv, tr)$
				\UnaryInf$\fCenter \Real{\lp^{\bot}_{\tau\rho}(t)}{}{y}, \Real{\lp^{\bot}_{\rho}(r)}{}{v}, \Real{\lp_{\tau}(tr)}{yv}{}$
			\end{bprooftree}
		\end{center}
		where in ($*$) we use Lemmas \ref{lem:embedPeanoIntoLinear} and \ref{closedTermsAreCp}, in ($\dagger$) we contract nonlinear formulas in the usual way, and in ($+$) we import this definition of constructing terms to $\ELPAomegal$ using Lemma \ref{lem:embedPeanoIntoLinear}.
		\end{enumerate}
		\item Linear axiom of choice:
		\begin{enumerate}[label=\alph*)]
			\item\mbox{}\\
			Consider the result of the following proof:
			\begin{prooftree}
				\AxiomC{}
				\RightLabel{(id)}
				\UnaryInf$\fCenter \lp(\alpha), \lp^{\bot}(\alpha)$
				\AxiomC{}
				\RightLabel{(id)}
				\UnaryInf$\fCenter \lp^{\bot}(Y), \lp(Y)$
				\RightLabel{($\otimes$)}
				\BinaryInf$\fCenter (\lp(\alpha) \lto \lp(Y))^{\bot}, \lp^{\bot}(\alpha), \lp(Y)$
				\AxiomC{}
				\RightLabel{(id)}
				\UnaryInf$\fCenter \place^{\bot}(v), \place(v)$
				\RightLabel{($\otimes$)}
				\BinaryInf$\fCenter (\lp(\alpha) \lto \lp(Y))^{\bot}, \lp^{\bot}(\alpha), \place^{\bot}(v), \lp(Y) \otimes \place(v)$
				\RightLabel{($\Par$)}
				\UnaryInf$\fCenter (\lp(\alpha) \lto \lp(Y))^{\bot}, (\lp(\alpha) \otimes \place(v))^{\bot}, \lp(Y) \otimes \place(v)$
				\RightLabel{($\Par$)}
				\UnaryInf$\fCenter (\lp(\alpha) \lto \lp(Y))^{\bot}, \lp(\alpha) \otimes \place(v) \lto \lp(Y) \otimes \place(v)$
			\end{prooftree}
			We apply the proven sequent via Lemma \ref{lem:substitute} to ($\lAC$), which yields
			\begin{multline*}
				\fCenter (\forall x^0 \exists y^0 \alpha x y =_0 0)^{\bot},\\
				\exists Y^1 (\forall x^0 (\alpha x (Yx) =_0 0) \otimes !(\lp(\alpha) \otimes \place(v) \lto \lp(Y) \otimes \place(v)))
			\end{multline*}
			and therefore
			\begin{multline*}
				\fCenter \Real{(\forall x^0 \exists y^0 \alpha x y =_0 0)^{\bot}}{}{},
				\Real{\exists Y^1 (\forall x^0 (\alpha x (Yx) =_0 0) \otimes !(\lp(\alpha)\lto \lp(Y)))}{\lambda v^0.v}{}
			\end{multline*}
			\item\mbox{}\\
			We have already shown that if we combine $\vdash (!(\lp(\alpha) \lto \lp(Y)))^{\bot}, 1$, which can be proven using ($1$) and (w?), with Corollary \ref{cor:substituteTensorPar} and ($\lAC$), we get an axiom of choice for terms:
			\begin{equation}
				\fCenter (\forall x^0 \exists y^0 \alpha x y =_0 0)^{\bot}, \exists Y^1 \forall x^0 (\alpha x (Yx) =_0 0)\period\tag{$*$}
			\end{equation}
			We assume, in $\EPAomega$, that $\forall x^0 \exists y^0 \alpha x y =_0 0$ holds and that $\con(z, \alpha')$ holds for some $z$ of appropriate type where $\alpha'xy =_0 0$ if and only if $y$ is the smallest such value with $\alpha x y =_0 0$. We use the following closed term:
			\begin{equation*}
				t :\equiv \lambda z, \lambda w^1, \lambda k^0. \left\{
				\begin{aligned}
					k_r + 1\phantom{0} &\text{ if $z w (\tilde{k_r}) k_l =_0 1$},\\
					0\phantom{k_r + 1} &\text{ otherwise.}
				\end{aligned}
				\right.
			\end{equation*}
			Let $Y'$ be a witness for $\alpha'$, i.e.~have the property $\alpha' x (Y'x) =_0 0$ for all $x^0$. Because of how we constructed $\alpha'$, there can only be one such $Y'$. By assuming $\con(w, x)$ for some $w^1$ and $x^0$, and showing $\con(tzw, Y'x)$, we will derive $\con(tz, Y')$. For this, let $k_r$ be a number such that $\alpha' x k_r =_0 0$ holds. Because of $\con(z, \alpha')$, there must be some $k_l$ with $z w (\tilde{k_r}) k_l =_0 1$. This implies $t z w k \neq_0 0$. Now, assume that the number $k_r$ is such that $t z w k =_0 k_r + 1$ holds. By definition, this implies $z w (\tilde{k_r}) k_l =_0 1$. Therefore, we have $\alpha' x k_r =_0 0$. Since $Y'$ is unique, we have $Y' x =_0 k_r$. This concludes the proof of $\con(tz, Y')$.
			
			Now, any witness $Y'$ for $\alpha'$ can also be a witness for $\alpha$. Moreover, from any witness $Y$ for $\alpha$, we can build a witness $Y'$ for $\alpha'$ by using $Y$ as bound for a search. We can therefore prove the following in $\EPAomega$:
			\begin{equation*}
				\forall x^0 \exists y^0 \alpha x y =_0 0 \vdash \exists {Y'}^1 (\forall x^0 (\alpha x (Y'x) =_0 0) \land (\con(z, \alpha) \to \con(tz, Y')))
			\end{equation*}
			where we used $Y'$ in the existential quantifier to indicate that we actually used the witness $Y'$ for $\alpha'$.
			We transport this proof to $\ELPAomegal$ using Lemma \ref{lem:embedPeanoIntoLinear} which produces a sequent that is identical to:
			\begin{equation*}
				\fCenter \Real{(\forall x^0 \exists y^0 \alpha x y =_0 0)^{\bot}}{}{}, \Real{\exists Y^1 (\forall x^0 (\alpha x (Yx) =_0 0) \otimes !(\lp(\alpha) \lto \lp(Y)))}{t}{}\period
			\end{equation*}
		\end{enumerate}
		\item Substitution for dot-equality:
		\begin{prooftree}
			\AxiomC{}
			\RightLabel{(dot-sub)}
			\UnaryInf$\fCenter (! x \doteq_{\tau} y)^{\bot}, (\Real{A}{\tupl{u}}{\tupl{v}})^{\bot}[\nicefrac{x}{z}], \Real{A}{\tupl{u}}{\tupl{v}}[\nicefrac{y}{z}]$
			\UnaryInf$\fCenter \Real{(! x \doteq_{\tau} y)^{\bot}}{}{}, \Real{A^{\bot}[\nicefrac{x}{z}]}{\tupl{v}}{\tupl{u}}, \Real{A[\nicefrac{y}{z}]}{\tupl{u}}{\tupl{v}}$
		\end{prooftree}
		\item Weakening (for $\EAPAomegaeq$):
		\begin{prooftree}
			\Axiom$\fCenter \Gamma$
			\RightLabel{(w)}
			\UnaryInf$\fCenter \Gamma, \Real{A}{\tupl{0}}{\tupl{v}}$
		\end{prooftree}
	\end{itemize}
\end{proof}

\begin{definition}[Linear quantifiers]
	We introduce the following abbreviations:
	\begin{align*}
		\forallL x^{\tau} A &:\equiv \forall x (\lp_{\tau}(x) \lto A)\\
		\existsL x^{\tau} A &:\equiv \exists x (\lp_{\tau}(x) \otimes A)\\
		\existsLeps x^{\tau} A &:\equiv \exists x (\lp_{\tau}(x) \otimes \epsilon =_0 0 \otimes A)
	\end{align*}
\end{definition}

\begin{remark}
	Both linear quantifiers $\forallL$ and $\existsL$ quantify only over \emph{linear} values, i.e.~ones that are visible to the functional interpretation.
	
	The special quantifier $\existsLeps$ is only of interest to our first theorem: This quantifier is ``enabled'' if and only if $\epsilon$ is equal to zero. Another way to look at it is that it carries some extra information and $\epsilon =_0 0$ stands for some arbitrary piece of such information. We use it to overcome the property of affine logic that allows us to assume premises without using them.
\end{remark}

\begin{theorem}\label{thm:firstCharacterization}
	Let $A(x^1)$, $B(x, y^1)$, $C(u^1)$, and $D(u, v^1)$ be formulas of $\EPAomega$ where $A$, $B$, $C$, and $D$ do not have any further free variables. Let $\Gamma$ be a set of formulas of the same language, and $\epsilon$ a variable different from $x$, $y$, $u$, and $v$.
	Consider the sequent
	\begin{equation}\tag{$*$}
		\fCenter \forallL x^1 (\emL{A}(x) \lto \existsLeps y^1 \emL{B}(x, y)) \lto \forallL u^1 (\emL{C}(u) \lto \existsLeps v^1 \emL{D}(u, v))\period
	\end{equation}
	The following are equivalent:
	\begin{enumerate}[label=\alph*)]
		\item $\ELPAomegal + \emL{\Gamma}$ proves ($*$).
		\item $\ELPAomegaeq + \emL{\Gamma}$ proves ($*$).
		\item $\EAPAomegal + \emL{\Gamma}$ proves ($*$).
		\item $\EAPAomegaeq + \emL{\Gamma}$ proves ($*$).
		\item $\EPAomega + \QFACnil + \Gamma$ proves both
		\begin{flalign*}
		&&&C(u) \to \halts{t \cdot u} \land A(t \cdot u)&\\
		\text{and }&&& C(u) \land B(t \cdot u, y) \to \halts{s \cdot \pair{u}{y}} \land D(u, s \cdot \pair{u}{y})&
		\end{flalign*}
		for some closed terms $t^1$ and $s^1$ of $\EPAomega$.\footnote{We read the formulas as $C(u) \to \halts{t \cdot u} \land (\wit(t \cdot u, x) \to A(x))$ and\\$C(u) \land \wit(t \cdot u, x) \land B(x, y) \to \halts{s \cdot \pair{u}{y}} \land (\wit(s \cdot \pair{u}{y}, v) \to D(u, v))$, respectively.}
	\end{enumerate}
\end{theorem}

\begin{proof}\mbox{}
	\begin{itemize}
		\item ``a) $\Rightarrow$ b)'' / ``a) $\Rightarrow$ c)'' / ``b) $\Rightarrow$ d)'' / ``c) $\Rightarrow$ d)'': We can copy any proof to a calculus with more rules.
		\item ``d) $\Rightarrow$ e)'': Assume that $\EAPAomegaeq + \emL{\Gamma}$ proves
		\begin{equation*}
		\fCenter \forallL x^1 (\emL{A}(x) \lto \existsLeps y^1 \emL{B}(x, y)) \lto \forallL u^1 (\emL{C}(u) \lto \existsLeps v^1 \emL{D}(u, v))\period
		\end{equation*}
		We apply Theorem \ref{Theorem:Dialectica} b) and yield terms $a$ and $b$ such that the following holds:
		\begin{equation*}
		\fCenter \Real{\forallL x^1 (\emL{A}(x) \lto \existsLeps y^1 \emL{B}(x, y)) \lto \forallL u^1 (\emL{C}(u) \lto \existsLeps v^1 \emL{D}(u, v))}{a, b}{f, g}\period
		\end{equation*}
		In the next few steps, we will make this more explicit by applying the definitions of our interpretation. In order to do this, we remove all abbreviations:
		\begin{multline*}
		\fCenter \Real{\exists x^1 (\lp(x) \otimes \emL{A}(x) \otimes \forall y^1 (\lp^{\bot}(y) \Par \epsilon \neq_0 0 \Par {\emL{B}}^{\bot}(x, y))) \\\Par \forall u^1 (\lp^{\bot}(u) \Par {\emL{C}}^{\bot}(u) \Par \exists v^1 (\lp(v) \otimes \epsilon =_0 0 \otimes \emL{D}(u, v)))}{a, b}{f, g}
		\end{multline*}
		We apply the definition of our interpretation for the first connective ($\Par$ in this case).
		\begin{multline*}
		\fCenter \Real{\exists x^1 (\lp(x) \otimes \emL{A}(x) \otimes \forall y^1 (\lp^{\bot}(y) \Par \epsilon \neq_0 0 \Par {\emL{B}}^{\bot}(x, y)))}{ag}{f} \\\Par \Real{\forall u^1 (\lp^{\bot}(u) \Par {\emL{C}}^{\bot}(u) \Par \exists v^1 (\lp(v) \otimes \epsilon =_0 0 \otimes \emL{D}(u, v)))}{bf}{g}\period
		\end{multline*}
		Since our interpretation is ``blind'' with respect to quantifiers and formulas that do not contain the linear predicate, we yield the following:
		\begin{multline*}
		\fCenter \exists x^1 (\Real{\lp(x)}{ag}{} \otimes \emL{A}(x) \otimes \forall y^1 ((\Real{\lp(y)}{f(ag)}{})^{\bot} \Par \epsilon \neq_0 0 \Par {\emL{B}}^{\bot}(x, y))) \\\Par \forall u^1 ((\Real{\lp(u)}{g}{})^{\bot} \Par {\emL{C}}^{\bot}(u) \Par \exists v^1 (\Real{\lp(v)}{bfg}{} \otimes \epsilon =_0 0 \otimes \emL{D}(u, v)))\period
		\end{multline*}
		Finally, we apply the interpretation of the linear predicate:
		\begin{multline*}
		\fCenter \exists x^1 (\emL{\con}(ag, x) \otimes \emL{A}(x) \otimes \forall y^1 (\emL{\con}(f(ag), y)^{\bot} \Par \epsilon \neq_0 0 \Par {\emL{B}}^{\bot}(x, y))) \\\Par \forall u^1 ((\emL{\con}(g, u))^{\bot} \Par {\emL{C}}^{\bot}(u) \Par \exists v^1 (\emL{\con}(bfg, v) \otimes \epsilon =_0 0 \otimes \emL{D}(u, v)))\period
		\end{multline*}
		Now, we use Lemma \ref{lem:embedLinearIntoPeano}
		in order to switch from linear logic to classical logic. We use this opportunity to simplify the formula:
		\begin{multline*}
		\fCenter \forall x^1 (\con(ag, x) \land A(x) \to \exists y^1 (\con(f(ag), y) \land \epsilon =_0 0 \land B(x, y))) \\\to \forall u^1 (\con(g, u) \land C(u) \to \exists v^1 (\con(bfg, v) \land \epsilon =_0 0 \land D(u, v)))\period
		\end{multline*}
		We remove both quantifiers $\exists y^1$ (this only strengthens the premise) and $\forall u^1$. Moreover, we replace $f$ and $g$ by $\lambda k^{1(1)}. \tilde{y}$ and $\tilde{u}$ respectively. Now, since with Lemma \ref{closedTermsAreCp} we know both $\con(\tilde{y}, y)$ and $\con(\tilde{u}, u)$, we can also remove both $\con(f(ag), y)$ and $\con(g, u)$.
		\begin{multline}\tag{$\dagger$}
		\fCenter \forall x^1 (\con(a\tilde{u}, x) \land A(x) \to \epsilon =_0 0 \land B(x, y)) \\\to (C(u) \to \exists v^1 (\con(b(\lambda k^{1(1)}. \tilde{y})\tilde{u}, v) \land \epsilon =_0 0 \land D(u, v)))\period
		\end{multline}
		We will use this sequent for two different values of $\epsilon$. We start with $\epsilon := 1$. Some simplification yields
		\begin{equation*}
		\forall x^1 (\con(a\tilde{u}, x) \land A(x) \to \bot), C(u) \vdash \bot\comma
		\end{equation*}
		which, using the law of excluded middle, lets us prove
		\begin{equation*}
		C(u) \vdash \exists x^1 (\con(a\tilde{u}, x) \land A(x))\period
		\end{equation*}
		If we apply Lemma~\ref{le:associateFacts}~\ref{leit:associateFactsReplaceConWithWit}, we can find a closed term $t$ and show the first sequent of our claim (using the fact that $x$ with $\con(a\tilde{u}, x)$ are unique):
		\begin{equation*}
		C(u) \vdash \halts{t \cdot u} \land (\wit(t \cdot u, x) \to A(x))\period
		\end{equation*}
		In order to prove the second sequent of our claim, we reuse ($\dagger$), this time with $\epsilon := 0$. We want to derive the left hand side, i.e.~$\forall x^1 (\con(a\tilde{u}, x) \land A(x)\to B(x, y))$, from the assumptions $\wit(t \cdot u, x)$ and $B(x, y)$. Therefore, we assume the latter formulas together with $\con(a\tilde{u}, x')$ and $A(x')$. Lemma~\ref{le:associateFacts}~\ref{leit:associateFactsReplaceConWithWit} tells us that $x$ and $x'$ are equal since both $\wit(t \cdot u, x)$ and $\con(a\tilde{u}, x')$ hold. We conclude $B(x', y)$. In summary, we can derive the following sequent:
		\begin{equation*}
		\wit(t \cdot u, x), B(x, y), \con(a\tilde{u}, x'), A(x') \vdash B(x', y)\period
		\end{equation*}
		Some applications of ($\to$R) and ($\forall$R) yield
		\begin{equation*}
		\wit(t \cdot u, x), B(x, y) \vdash \forall x (\con(a\tilde{u}, x) \land A(x) \to B(x, y))\period
		\end{equation*}
		If we cut this with ($\dagger$) for $\epsilon := 0$ (and rearrange some things), we get the following:
		\begin{equation*}
		C(u), \wit(t \cdot u, x), B(x, y) \vdash \exists v^1 (\con(b(\lambda k^{1(1)}. \tilde{y})\tilde{u}, v) \land D(u, v))\period
		\end{equation*}
		Finally, if we apply Lemma~\ref{le:associateFacts}~\ref{leit:associateFactsReplaceConWithWit} to this, we can produce a closed term $s$ and prove the second sequent of our claim:
		\begin{equation*}
		C(u), \wit(t \cdot u, x), B(x, y) \vdash \halts{s \cdot \pair{u}{y}} \land (\wit(s \cdot \pair{u}{y}, v) \to D(u, v))\period
		\end{equation*}
		\item ``e) $\Rightarrow$ a)'':
		Assume that $\EPAomega + \QFACnil + \Gamma$ proves both formulas from e). We apply Lemma~\ref{lem:embedPeanoIntoLinear} and can therefore prove the following sequents in $\ELPAomegal + \emL{\Gamma}$:
		\begin{align*}
		&\fCenter {\emL{C}}^{\bot}(u), \emL{(\halts{t \cdot u})}\comma\tag{i}\\
		&\fCenter {\emL{C}}^{\bot}(u), {\emL{\wit}}^{\bot}(t \cdot u, x), \emL{A}(x)\comma\tag{ii}\\
		&\fCenter {\emL{C}}^{\bot}(u), {\emL{\wit}}^{\bot}(t \cdot u, x), {\emL{B}}^{\bot}(x, y), \emL{(\halts{s \cdot \pair{u}{y}})}\comma\tag{iii}\\
		&\fCenter {\emL{C}}^{\bot}(u), {\emL{\wit}}^{\bot}(t \cdot u, x), {\emL{B}}^{\bot}(x, y), {\emL{\wit}}^{\bot}(s \cdot \pair{u}{y}, v), \emL{D}(u, v)\period\tag{iv}
		\end{align*}
		We start with (iv). We apply (id) and ($\otimes$) for both $\lp(v)$ and $\epsilon =_0 0$, use ($\Par$), ($\exists$), and ($\forall$) in order to derive
		\begin{multline*}
		\fCenter {\emL{C}}^{\bot}(u), {\emL{\wit}}^{\bot}(t \cdot u, x), (\epsilon =_0 0)^{\bot}, {\emL{B}}^{\bot}(x, y),\\ (\existsL v\ \emL{\wit}(s \cdot \pair{u}{y}, v))^{\bot}, \existsLeps v \emL{D}(u, v)\period
		\end{multline*}
		We cut this with an instance of Lemma~\ref{le:wit}~\ref{leit:witIntroC}:
		\begin{multline*}
		\fCenter \lp^{\bot}(u), {\emL{C}}^{\bot}(u), {\emL{\wit}}^{\bot}(t \cdot u, x), \lp^{\bot}(y), (\epsilon =_0 0)^{\bot}, {\emL{B}}^{\bot}(x, y),\\ {\emL{(\halts{s \cdot \pair{u}{y}})}}^{\bot}, \existsLeps v \emL{D}(u, v)
		\end{multline*}
		Here, we split $\lp(\pair{u}{y})$ into $\lp(j)$, $\lp(u)$, and $\lp(y)$ by cutting with ($\lp$-app) twice, and we used that both $s$ and $j$ are closed terms. We cut the sequent with (iii). Recall that we can contract nonlinear formulas that occur twice using (d$?$), (c$?$), and a cut with ($!_2$).
		\begin{equation*}
		\fCenter \lp^{\bot}(u), {\emL{C}}^{\bot}(u), {\emL{\wit}}^{\bot}(t \cdot u, x), \lp^{\bot}(y), (\epsilon =_0 0)^{\bot}, {\emL{B}}^{\bot}(x, y), \existsLeps v \emL{D}(u, v)
		\end{equation*}
		With ($\Par$) twice and ($\forall$) we can prove
		\begin{equation*}
		\fCenter \lp^{\bot}(u), {\emL{C}}^{\bot}(u), {\emL{\wit}}^{\bot}(t \cdot u, x), (\existsLeps y \emL{B}(x, y))^{\bot}, \existsLeps v \emL{D}(u, v)\period
		\end{equation*}
		We apply ($\otimes$) to this sequent with (ii). Similar to before, we can contract both occurrences of ${\emL{C}}^{\bot}(u)$ and ${\emL{\wit}}^{\bot}(t \cdot u, x)$:
		\begin{equation*}
		\fCenter \lp^{\bot}(u), {\emL{C}}^{\bot}(u), {\emL{\wit}}^{\bot}(t \cdot u, x), (\emL{A}(x) \lto \existsLeps y \emL{B}(x, y))^{\bot}, \existsLeps v \emL{D}(u, v)\period
		\end{equation*}
		We use ($\otimes$) for a second time, this time with the axiom $\fCenter \lp^{\bot}(x), \lp(x)$:
		\begin{multline*}
		\fCenter \lp^{\bot}(u), {\emL{C}}^{\bot}(u), \lp^{\bot}(x), {\emL{\wit}}^{\bot}(t \cdot u, x),\\ (\lp^{\bot}(x) \Par (\emL{A}(x) \lto \existsLeps y \emL{B}(x, y)))^{\bot}, \existsLeps v \emL{D}(u, v)
		\end{multline*}
		Some applications of ($\exists$), ($\Par$), and ($\forall$) yield
		\begin{multline*}
		\fCenter \lp^{\bot}(u), {\emL{C}}^{\bot}(u), (\existsL x\ \emL{\wit}(t \cdot u, x))^{\bot},\\ (\forallL x (\emL{A}(x) \lto \existsLeps y \emL{B}(x, y)))^{\bot}, \existsLeps v \emL{D}(u, v)
		\end{multline*}
		We cut with an instance of Lemma~\ref{le:wit}~\ref{leit:witIntroC} for a second time. Again, we can cut $\lp^{\bot}(t)$ away since $t$ is closed, and contract both ensuing occurrences of $\lp^{\bot}(u)$.
		\begin{equation*}
		\fCenter \lp^{\bot}(u), {\emL{C}}^{\bot}(u), {\emL{(\halts{t \cdot u})}}^{\bot}, (\forallL x (\emL{A}(x) \lto \existsLeps y \emL{B}(x, y)))^{\bot}, \existsLeps v \emL{D}(u, v)
		\end{equation*}
		We cut this sequent with (i) and contract both occurrences of ${\emL{C}}^{\bot}(u)$:
		\begin{equation*}
		\fCenter \lp^{\bot}(u), {\emL{C}}^{\bot}(u), (\forallL x (\emL{A}(x) \lto \existsLeps y \emL{B}(x, y)))^{\bot}, \existsLeps v \emL{D}(u, v)
		\end{equation*}
		Finally, applications of ($\Par$) and ($\forall$) lead us to our claim
		\begin{equation*}
		\fCenter \forallL x (\emL{A}(x) \lto \existsLeps y \emL{B}(x, y)) \lto \forallL u (\emL{C}(u) \lto \existsLeps v \emL{D}(u, v))\period
		\end{equation*}
	\end{itemize}
\end{proof}

\section{Phase Spaces and the Extraction of Argument Paths}\label{sec:Phase}

In this section, we will improve our previous result: We replace both quantifiers $\existsLeps$ by $\existsL$, which makes the characterization more symmetrical and natural. Even if the verifying system was linear, the functional interpretation ultimately constructs terms for a classical system where the information for the termination of the first Weihrauch program is not retrievable anymore. In order to show that this program does indeed halt, as long as we can provide the proof in linear instead of affine logic, we employ a semantics defined by Girard that rejects the weakening rule: \emph{phase semantics}. We will use a simplified version that suffices for our purposes. Girard's original definitions and results, including the more general Soundness Proposition, can be found in \cite[pp.~17--28]{girard1987}.

\begin{definition}[Phase space]
	We call the multiplicative monoid $P := \{0, 1\}$ together with its set of \emph{antiphases} $\bot := \{1\}$ a \emph{phase space}.
\end{definition}
For our purposes, this particular phase space suffices. The more general definition requires $P$ to be any commutative monoid and $\bot$ to be any of its subsets.

\begin{definition}
	For any subset $Q$ of $P$, we define
	\begin{equation*}
	Q^{\bot} := \{p \in P: \forall q \in Q \ pq \in \bot\}\period
	\end{equation*}
\end{definition}
In our case, we have: $\emptyset^{\bot} = \{0, 1\}$, $\{0\}^{\bot} = \emptyset$, $\{1\}^{\bot} = \{1\}$, and $\{0, 1\}^{\bot} = \emptyset$.

\begin{definition}
	A subset $Q$ of $P$ with the property $Q^{\bot\bot} = Q$ is called a \emph{fact}. The elements of $Q$ are the \emph{phases} of $Q$, and $Q$ is a \emph{valid fact} iff $1$ is an element of $Q$.
\end{definition}
In our setting, we have three facts: $\emptyset$, $\{1\}$, and $\{0,1\}$. Only $\{1\}$ and $\{0,1\}$ are valid facts. For $\{0\}$, we have $\{0\}^{\bot\bot} = \emptyset^{\bot} = \{0, 1\} \neq \{0\}$. Therefore, $\{0\}$ is not a fact.
In order to talk about the standard model of arithmetic in all finite types, we introduce a variable assignment that takes a variable of some type and gives an object of the same type.

\begin{definition}[Variable assignment]
	We call a function
	\begin{equation*}
	\beta: \text{Var} \to \bigcup_{\tau} T^{\tau}
	\end{equation*}
	where for every type $\tau$ the set $T^{\tau}$ consists of all closed terms in $\EPAomega$ of that type a \emph{variable assignment} if it maps every variable $x^{\tau}$ of type $\tau$ to a closed term $\beta(x)$ of this very type. We may also write $\beta(t^{\tau})$ for the closed term where all variables in $t$ are replaced by the values that $\beta$ maps them to.
\end{definition}

For the semantics of linear logic without modalities, it suffices to consider phase spaces. For modalities, Girard uses \emph{topolinear spaces} in his original work (cf.~\cite[p.~25]{girard1987}) although there seems to be a revised definition of phase semantics that captures modalities in a simplified way without introducing topolinear spaces (cf.~\cite[p.~198]{girardNeu}). In order to keep our proof simple, we will avoid topolinear spaces altogether, and define slightly different semantics that already capture everything we need for our purposes.

\begin{definition}[Simple phase semantics for $\ELPAomegaeq$]
	We define a way to map every sequent of $\ELPAomegaeq$ to one of the three facts. We start by defining the semantics for atomic formulas and some variable assignment $\beta$:
	\begin{alignat*}{3}
	&\Sem{\Aat}{\beta} &&:= \{1\} \text{ for $\Aat \in \{1, s =_0 t, \place(t) \text{ with $\beta(t) = 0$}\}$}\comma\\
	&\Sem{\place(t)}{\beta} &&:= \emptyset \text{ or  $\{0, 1\}$ for $\beta(t) > 0$}\comma\\
	&\Sem{\top}{\beta} &&:= \{0, 1\}\comma\\
	&\Sem{s \doteq_0 t}{\beta} &&:= \left\{
	\begin{aligned}
	\{0, 1\}\phantom{\emptyset} & \text{ if $\beta(s) = \beta(t)$}\comma\\
	\emptyset\phantom{\{0, 1\}} & \text{ otherwise}\comma\\
	\end{aligned}\right.\\
	&\Sem{\Aat^{\bot}}{\beta} &&:= \Sem{\Aat}{\beta}^{\bot} \text{ for atomic $\Aat$}\period\\
	\intertext{We assume that $\Sem{A}{\beta}$ and $\Sem{B}{\beta}$ have already been defined for some formulas $A$ and $B$:}
	&\Sem{A \otimes B}{\beta} &&:= \{pq: p \in \Sem{A}{\beta} \text{ and } q \in \Sem{B}{\beta}\}\comma\\
	&\Sem{A \with B}{\beta} &&:= \Sem{A}{\beta} \cap \Sem{B}{\beta}\comma\\
	&\Sem{?A}{\beta} &&:= \Sem{A}{\beta} \cup \{1\}\comma\\
	&\Sem{\forall x^{\tau} A}{\beta} &&:= \bigcap_{s \in S^{\tau}} \Sem{A}{\beta[\nicefrac{s}{x}]}\comma\\
	&\Sem{A \Par B}{\beta} &&:= \Sem{A^{\bot} \otimes B^{\bot}}{\beta}^{\bot}\comma\\
	&\Sem{A \oplus B}{\beta} &&:= \Sem{A^{\bot} \with B^{\bot}}{\beta}^{\bot}\comma\\
	&\Sem{!A}{\beta} &&:= \Sem{?A^{\bot}}{\beta}^{\bot}\comma\\
	&\Sem{\exists x^{\tau} A}{\beta} &&:= \Sem{\forall x^{\tau} A^{\bot}}{\beta}^{\bot}\period
	\end{alignat*}
\end{definition}

\begin{remark}
	Similar to Lemma \ref{Lemma:InvolutionInterpretation}, we can easily show $\Sem{A^{\bot}}{\beta} = \Sem{A}{\beta}^{\bot}$ for all formulas $A$ and variable assignments $\beta$.
\end{remark}

Notice how the semantics $\Sem{\place(t)}{\beta}$ for $\beta(t) > 0$ are only restricted to the facts $\emptyset$ or $\{0, 1\}$ for now. Both possible definitions are sound, at least if they are fixed beforehand and not changed during the Soundness proof, and we will need both of them during the proof of our final theorem.

\begin{definition}
	Let $\fCenter A_1, \dots, A_n$ be a sequent. We say that this sequent holds semantically (with respect to $P$), if and only if for each variable assignment $\beta$ one of the two following properties holds:
	\begin{itemize}
		\item $\Sem{A_i}{\beta} = \{1\}$ for all indices $i$ with $1 \leq i \leq n$,
		\item $\Sem{A_i}{\beta} = \{0,1\}$ for some index $i$ with $1 \leq i \leq n$.
	\end{itemize}
	In this case, we write $\Vdash A_1, \dots, A_n$.
\end{definition}
\begin{remark}
	It can easily be seen that one of the two properties holds if and only if the set $\bigcap_{\beta}\Sem{A_1 \Par \dots \Par A_n}{\beta}$ contains $1$. While this is the more general definition that also works for phase spaces different from our particular $P$ (cf.~Definition 1.14 in \cite[p.~22]{girard1987}), we chose the previous one because it will keep the Soundness proof shorter.
\end{remark}

The following lemma is a simplified version of Girard's Soundness result (cf.~Proposition 1.16 in \cite[p.~23]{girard1987}) for $\ELPAomegaeq$:

\begin{lemma}\label{Lemma:PhaseSoundness}
	Let $\Gamma$ be a set of formulas such that $\EPAomega + \Gamma + \QFACnil$ is consistent.	If $\ELPAomegaeq + \emL{\Gamma}$ proves the sequent $\vdash \Delta$, then it holds semantically with respect to $P$, i.e.~$\Vdash \Delta$.
\end{lemma}

\begin{proof}
	We will prove by induction on the length of the proof for $\vdash \Delta$ that $\Vdash \Delta$ holds. For the first steps, we will assume that $\beta$ is an arbitrary variable assignment.
	\begin{itemize}
		\item Axioms from $\emL{\Gamma}$: The translation from Definition \ref{Definition:Embedding} only produces formulas $A$ with $\Sem{A}{\beta} = \{1\}$. We therefore immediately have $\Vdash A$.
		\item (id): Assume that one of $\Sem{A}{\beta}$ and $\Sem{A^{\bot}}{\beta}$ is different from $\{1\}$. Then, w.l.o.g., we can assume that $\Sem{A}{\beta} = \{0, 1\}$ holds.
		\item (cut): If one of the sequents $\Gamma$ or $\Delta$ contains a formula $B$ with $\Sem{B}{\beta} = \{0, 1\}$, we are finished. Otherwise, if $\Sem{A}{\beta} = \{0, 1\}$, we conclude $\Sem{A^{\bot}}{\beta} = \emptyset$ and, therefore, that $\Delta$ must contain some $D$ with $\Sem{D}{\beta} = \{0, 1\}$. An analogous argument works for $\Sem{A}{\beta} = \emptyset$. Finally, suppose $\Sem{A}{\beta} = \{1\}$. In this case, all formulas $B$ in $\Gamma, \Delta$ must have the property $\Sem{B}{\beta} = \{1\}$.
		\item (per): This is trivial, since neither of our conditions depends on the order of formulas.
		\item ($\otimes$): If one of the sequents $\Gamma$ and $\Delta$ contains a formula $C$ with $\Sem{C}{\beta} = \{0, 1\}$, we are finished. Otherwise, we know that both $\Sem{A}{\beta}$ and $\Sem{B}{\beta}$ must contain $1$. Should any of them be equal to $\{0, 1\}$, then we have $\Sem{A \otimes B}{\beta} = \{0, 1\}$. Finally, should both be equal to $\{1\}$, then we also have $\Sem{A \otimes B}{\beta} = \{1\}$ and the same holds for all $\Sem{C}{\beta}$ for formulas $C$ in $\Gamma, \Delta$.
		\item ($\Par$): If the sequent contains a formula $C$ with $\Sem{C}{\beta} = \{0, 1\}$, then we are finished: This is clear if $C$ lives in $\Gamma$, and should any of $A$ and $B$ have this property, then $\Sem{A \Par B}{\beta} = \Sem{A^{\bot} \otimes B^{\bot}}{\beta}^{\bot} = \emptyset^{\bot} = \{0, 1\}$ holds. Finally, if for all $C$ in the sequent we have $\Sem{C}{\beta} = \{1\}$, then we conclude $\Sem{A \Par B}{\beta} = \Sem{A^{\bot} \otimes B^{\bot}}{\beta}^{\bot} = \{1\}^{\bot} = \{1\}$.
		\item (1): It follows directly, since $\Sem{1}{\beta} = \{1\}$ holds.
		\item ($\bot$): The addition of $\bot$ to the sequent keeps the invariant intact, since $\Sem{\bot}{\beta} = \{1\}$ holds.
		\item ($\with$): If any formula $C$ in $\Gamma$ has the property $\Sem{C}{\beta} = \{0, 1\}$, we are finished. Otherwise, if there exists such a $C$ with $\Sem{C}{\beta} = \emptyset$, then we can conclude both $\Sem{A}{\beta} = \{0, 1\}$ and $\Sem{B}{\beta} = \{0, 1\}$. This yields $\Sem{A \with B}{\beta} = \{0, 1\} \cap \{0, 1\} = \{0, 1\}$. Finally, if all formulas $C$ in $\vdash \Gamma, A, B$ have $\Sem{C}{\beta} = \{1\}$, then the semantics for all formulas in $\vdash \Gamma, A \with B$ contains $1$.
		\item ($\oplus$): Since the semantics of $\oplus$ are commutative, we only consider ($\oplus_1$). If $\Sem{C}{\beta} = \{0, 1\}$ holds for any $C$ in $\Gamma$, we are finished. Otherwise, if $\Sem{A}{\beta} = \{0, 1\}$, then we have $\Sem{A \oplus B}{\beta} = \Sem{A^{\bot} \with B^{\bot}}{\beta}^{\bot} = (\emptyset \cap \Sem{B^{\bot}}{\beta})^{\bot} = \emptyset^{\bot} = \{0, 1\}$. Finally, should all formulas $C$ in $\vdash \Gamma, A$ have the property $\Sem{C}{\beta} = \{1\}$, then the evaluation of all formulas in $\vdash \Gamma, A \oplus B$ includes $1$ since $\Sem{A \oplus B}{\beta} = \Sem{A^{\bot} \with B^{\bot}}{\beta}^{\bot} = (\{1\} \cap \Sem{B^{\bot}}{\beta})^{\bot} \supseteq \{1\}^{\bot} = \{1\}$.
		\item ($\top$): The addition of $\top$ to any sequent immediately establishes the invariant, since $\Sem{\top}{\beta} = \{0, 1\}$ holds.
		\item (!): All formulas $B$ in $?\Gamma$ have the property $\Sem{B}{\beta} \supseteq \{1\}$. Should any of these be equal to $\{0, 1\}$, then we are finished. Otherwise, we know that $\Sem{A}{\beta} \supseteq \{1\}$ must hold. This yields $\Sem{!A}{\beta} = \Sem{?A^{\bot}}{\beta}^{\bot} = \{1\}^{\bot} = \{1\}$.
		\item (d?): Because of $\Sem{?A}{\beta} = \Sem{A}{\beta} \cup \{1\}$, this rule cannot break the invariant.
		\item (w?): Since we have $\Sem{?A}{\beta} \supseteq \{1\}$, the addition of $?A$ to our sequent keeps the invariant intact.
		\item (c?): Both properties still hold after removing one duplicate formula from the sequent.
		\item ($\forall$): Our three facts have the property $\emptyset \subset \{1\} \subset \{0, 1\}$. Hence, for each variable assignment $\beta$ there exists some $s \in S^{\tau}$ with $\Sem{A}{\beta[\nicefrac{s}{x^{\tau}}]} = \bigcap_{s \in S^{\tau}} \Sem{A}{\beta[\nicefrac{s}{x}]} = \Sem{\forall x^{\tau} A}{\beta}$. Since $x$ does not occur anywhere else freely in the sequent, the semantics of $\vdash \Gamma, \forall x^{\tau} A$ with respect to $\beta$ are identical to that of $\vdash \Gamma, A$ with respect to $\beta[\nicefrac{s}{x}]$. We can therefore conclude $\Vdash \Gamma, \forall x^{\tau} A$ from $\Vdash \Gamma, A$.
		\item ($\exists$): Under the induction hypothesis, the invariant stays intact if the set given by the semantics of some formula gets bigger: $\Sem{\exists x^{\tau} A}{\beta} = \Sem{\forall x^{\tau} A^{\bot}}{\beta}^{\bot} \supseteq \Sem{A^{\bot}}{\beta[\nicefrac{\beta(t)}{x}]}^{\bot} = \Sem{A}{\beta[\nicefrac{\beta(t)}{x}]} = \Sem{A[\nicefrac{t}{x}]}{\beta}$. In the last equation, we used that $t$ is free for $x$ in $A$.
		\item Axioms for equality, projectors, combinators, recursors, induction, ($\lp$), ($!_2$), ($\lp$-app), and ($\lAC$) all consist of formulas $A$ with $\Sem{A}{\beta} = \{1\}$ for all $\beta$. Notice that ($!_2$) is restricted to nonlinear formulas that do neither contain dot-equality nor the tagging predicate.
		\item (dot-succ): $\Sem{!Sx \not\doteq_0 0}{\beta} = \{1\}$.
		\item (dot-eq$_1$): If we can prove $\vdash s =_{\tau} t$ in $\ELPAomegaeq + \emL{\Gamma}$, we can transform it into a derivation of the same sequent in $\EPAomega + \QFACnil + \Gamma$ via Lemma \ref{lem:embedLinearIntoPeano}. The equality $s\tupl{x} =_0 t\tupl{x}$ is provable for any substitution of the free variables by closed terms where $\tupl{x}$ is a tuple of fresh variables such that both $s\tupl{x}$ and $t\tupl{x}$ are of type $0$. These equalities of numbers and closed terms must also hold in the standard model, since our system is consistent by assumption. We conclude $\Sem{s\tupl{x} \doteq_0 t\tupl{x}}{\beta} = \{0,1\}$ for all $\beta$ and therefore $\Sem{!s \doteq_{\tau} t}{\beta} = \{1\}$.
		\item (dot-eq$_2$): Assume $\beta(k) = 0$. Then we have $\Sem{k \not\doteq_0 0}{\beta} = \emptyset$, $\Sem{!k \not\doteq_0 0}{\beta} = \emptyset$, and $\Sem{(!k \not\doteq_0 0)^{\bot}}{\beta} = \{0, 1\}$. This means that the second case of our invariant holds. Otherwise, if $\beta(k)$ is different from zero, we have $\beta(\sg(k)) = 1$, $\Sem{\sg(k) \doteq_0 1}{\beta} = \{0, 1\}$, and $\Sem{!\sg(k) \doteq_0 1}{\beta} = \{1\}$. Moreover, $\Sem{k \not\doteq_0 0}{\beta} = \{0, 1\}$, $\Sem{!k \not\doteq_0 0}{\beta} = \{1\}$, and $\Sem{(!k \not\doteq_0 0)^{\bot}}{\beta} = \{1\}$. This means that the first case of our invariant holds.
		\item (dot-sub): Assume that $x \neq y$ holds with respect to $\beta$. Then we have $\Sem{(!x \doteq_{\tau} y)^{\bot}}{\beta} = \Sem{?(x \doteq y)^{\bot}}{\beta} = \{0, 1\}$. Otherwise, in the case of $x = y$, we conclude $\Sem{A^{\bot}[\nicefrac{x}{z}]}{\beta} = (\Sem{A[\nicefrac{y}{z}]}{\beta})^{\bot}$. Now, the argument is really similar to (id): If $\Sem{A[\nicefrac{y}{z}]}{\beta}$ is equal to $\emptyset$ or to $\{0, 1\}$, we are immediately finished. Finally, if both $\Sem{A[\nicefrac{y}{z}]}{\beta}$ and $\Sem{A^{\bot}[\nicefrac{x}{z}]}{\beta}$ are equal to $\{1\}$, we have to show that the semantics of the equality contain $1$: $\Sem{(!x \doteq_{\tau} y)^{\bot}}{\beta} = \Sem{?(x \doteq y)^{\bot}}{\beta} \supseteq \{1\}$.
		\item ($\lp$-con) and (tag-con): Contracting any two duplicate formulas in a sequent keeps the invariant intact.
		\item (tag-0): By definition $\Sem{\place(0)}{\beta} = \{1\}$.
		\item (tag-app): Assume $\beta(x) = \beta(y) = 0$. In this case, we have $\Sem{\place^{\bot}(\sg(x))}{\beta} = \Sem{\place^{\bot}(\sg(y))}{\beta} = \Sem{\place(\sg(x + y))}{\beta} = \{1\}$. Otherwise, w.l.o.g. let $\beta(x) > 0$. This implies $\beta(\sg(x)) = \beta(\sg(x + y)) = 1$. We conclude that either $\Sem{\place^{\bot}(\sg(x))}{\beta}$ or $\Sem{\place(\sg(x + y))}{\beta}$ is equal to $\{0, 1\}$.
	\end{itemize}
\end{proof}

\begin{theorem}\label{thm:secondCharacterization}
	Let $A(x^1)$, $B(x, y^1)$, $C(u^1)$, and $D(u, v^1)$ be formulas of $\EPAomega$ where $A$, $B$, $C$, and $D$ do not have any further free variables. Let $\Gamma$ be a set of formulas of the same language. Then the following are equivalent:
	\begin{enumerate}[label=\alph*)]
		\item $\ELPAomegal + \emL{\Gamma}$ proves the sequent
		\begin{equation*}
		\fCenter \forallL x^1 (\emL{A}(x) \lto \existsL y^1 \emL{B}(x, y)) \lto \forallL u^1 (\emL{C}(u) \lto \existsL v^1 \emL{D}(u, v))\period
		\end{equation*}
		\item $\EPAomega + \QFACnil + \Gamma$ proves both
		\begin{flalign*}
		&&&C(u) \to \halts{t \cdot u} \land A(t \cdot u)&\\
		&\text{and }&& C(u) \land B(t \cdot u, y) \to \halts{s \cdot \pair{u}{y}} \land D(u, s \cdot \pair{u}{y})&
		\end{flalign*}
		for some closed terms $t^1$ and $s^1$ of $\EPAomega$.
	\end{enumerate}
\end{theorem}

\begin{proof}
	The direction ``b) $\Rightarrow$ a)'' works like Theorem \ref{thm:firstCharacterization} ``e) $\Rightarrow$ a)'' if we omit the introduction of $\epsilon =_0 0$. For ``a) $\Rightarrow$ b)'', assume that $\ELPAomegal + \emL{\Gamma}$ proves
	\begin{equation*}
		\fCenter \forallL x^1 (\emL{A}(x) \lto \existsL y^1 \emL{B}(x, y)) \lto \forallL u^1 (\emL{C}(u) \lto \existsL v^1 \emL{D}(u, v))\comma
	\end{equation*}
	we also know that $\ELPAomegaeq + \emL{\Gamma}$ proves the same sequent. We can, therefore, apply Theorem \ref{Theorem:Dialectica} a), where we interpret $\Real{\lp_{\tau}(t^{\tau})}{x^0}{} :\equiv \lp_{\tau}(t) \otimes \place(\sg(x))$. This yields terms $a$ and $b$ such that we can show the following sequent in $\ELPAomegaeq + \Gamma$:
	\begin{equation*}
	\fCenter \Real{\forallL x^1 (\emL{A}(x) \lto \existsL y^1 \emL{B}(x, y)) \lto \forallL u^1 (\emL{C}(u) \lto \existsL v^1 \emL{D}(u, v))}{a, b}{f, g}
	\end{equation*}
	Except for the types, the following steps are identical to those of the proof for the direction ``b) $\Rightarrow$ c)'' of Theorem \ref{thm:firstCharacterization}. We replace all abbreviations
	\begin{multline*}
	\fCenter \Real{\exists x^1 (\lp(x) \otimes \emL{A}(x) \otimes \forall y^1 (\lp^{\bot}(y) \Par {\emL{B}}^{\bot}(x, y))) \\\Par \forall u^1 (\lp^{\bot}(u) \Par {\emL{C}}^{\bot}(u) \Par \exists v^1 (\lp(v) \otimes \emL{D}(u, v)))}{a, b}{f, g}
	\end{multline*}
	and apply the definition of our interpretation:
	\begin{multline*}
	\fCenter \Real{\exists x^1 (\lp(x) \otimes \emL{A}(x) \otimes \forall y^1 (\lp^{\bot}(y) \Par {\emL{B}}^{\bot}(x, y)))}{ag}{f} \\\Par \Real{\forall u^1 (\lp^{\bot}(u) \Par {\emL{C}}^{\bot}(u) \Par \exists v^1 (\lp(v) \otimes \emL{D}(u, v)))}{bf}{g}
	\end{multline*}
	Since our interpretation is ``blind'' with respect to quantifiers and formulas that do not contain the linear predicate, our sequent is identical to the following:
	\begin{multline*}
	\fCenter \exists x^1 (\Real{\lp(x)}{ag}{} \otimes \emL{A}(x) \otimes \forall y^1 ((\Real{\lp(y)}{f(ag)}{})^{\bot} \Par {\emL{B}}^{\bot}(x, y))) \\\Par \forall u^1 ((\Real{\lp(u)}{g}{})^{\bot} \Par {\emL{C}}^{\bot}(u) \Par \exists v^1 (\Real{\lp(v)}{bfg}{} \otimes \emL{D}(u, v)))
	\end{multline*}
	Now, we put in the interpretation of the linear predicate:
	\begin{multline*}
	\fCenter \exists x^1 (\lp(x) \otimes \place(\sg(ag)) \otimes \emL{A}(x) \otimes \forall y^1 (\lp^{\bot}(y) \Par \place^{\bot}(\sg(f(ag))) \Par {\emL{B}}^{\bot}(x, y))) \\\Par \forall u^1 (\lp^{\bot}(u) \Par \place^{\bot}(\sg(g)) \Par {\emL{C}}^{\bot}(u) \Par \exists v^1 (\lp(v) \otimes \place(\sg(bfg)) \otimes \emL{D}(u, v)))
	\end{multline*}
	We replace the variable $f$ by $\lambda k^0. 1$ and $g$ by $0$. Moreover, we introduce the abbreviations $a' :\equiv \sg(a0)$ and $b' :\equiv \sg(b(\lambda k^0.1)0)$.
	
	Since we can prove the equalities $\sg(0) =_0 0$ and $\sg((\lambda k^0. 1)(ag)) =_0 1$ in $\EPAomega$, we can replace $\place^{\bot}(\sg(0))$ and $\place^{\bot}(\sg((\lambda k^0. 1)(ag)))$ in our sequent by $\place^{\bot}(0)$ and $\place^{\bot}(1)$, respectively, using Lemma \ref{lem:embedPeanoIntoLinear}, (dot-eq$_1$), and (dot-sub).
	\begin{multline}
	\fCenter \exists x^1 (\lp(x) \otimes \place(a') \otimes \emL{A}(x) \otimes \forall y^1 (\lp^{\bot}(y) \Par \place^{\bot}(1) \Par {\emL{B}}^{\bot}(x, y))) \\\Par \forall u^1 (\lp^{\bot}(u) \Par \place^{\bot}(0) \Par {\emL{C}}^{\bot}(u) \Par \exists v^1 (\lp(v) \otimes \place(b') \otimes \emL{D}(u, v)))\tag{$\dagger$}
	\end{multline}
	For now, we are only interested in semantics, i.e.~we apply Lemma \ref{Lemma:PhaseSoundness}. If\\$\EPAomega + \QFACnil + \emL{\Gamma}$ is not consistent, which is a requirement for this lemma, then the claim b) follows immediately by \emph{ex falso quodlibet}.
	
	We know, that all formulas that neither contain dot-equality nor the tag predicate evaluate to $\{1\}$. Similarly, the semantics of $\place^{\bot}(0)$ result in the same set $\{1\}$. Furthermore, this set acts as a neutral element with respect to the connectives $\otimes$ and $\Par$. Finally, the quantifiers do nothing, since the subformulas evaluate to the same sets regardless of any variable assignments. Hence, we can simplify the formula by a lot:
	\begin{equation*}
	\Vdash (\place(a') \otimes \place^{\bot}(1)) \Par \place(b')
	\end{equation*}
	Now comes the heart of the proof: We want to show that $a'$ and $b'$ evaluate to $0$ and $1$, respectively. If we can show this, we know that the derivation of the sequent happened in very particular way. We show this by case distinction. This is possible, since we did not fix the semantics of $\place(x)$ for positive $x$.
	
	There are only four possibilities for the values of $a'$ and $b'$ since both of them have $\sg$ as outermost term which can only map to $0$ or $1$:
	\begin{itemize}
		\item $a' = 0$ and $b' = 0$: In this case, we choose $\Sem{\place(x)}{\beta} = \{0, 1\}$ for $x > 0$. This implies $\Sem{\place(a') \otimes \place^{\bot}(1)}{\beta} = \emptyset$ and, because of $\Sem{\place(b')}{\beta} = \{1\}$, also $\Sem{(\place(a') \otimes \place^{\bot}(1)) \Par \place(b')}{\beta} = \emptyset$.
		\item $a' = 1$ and $b' = 0$: Here, both choices $\emptyset$ and $\{0, 1\}$ for $\Sem{\place(x)}{\beta}$ with $x > 0$ would work. We choose $\emptyset$: Similar to above, we have $\Sem{\place(a') \otimes \place^{\bot}(1)}{\beta} = \emptyset$ and, because of $\Sem{\place(b')}{\beta} = \{1\}$, also $\Sem{(\place(a') \otimes \place^{\bot}(1)) \Par \place(b')}{\beta} = \emptyset$.
		\item $a' = 1$ and $b' = 1$: Finally, we choose $\Sem{\place(x)}{\beta} = \emptyset$ for $x > 0$. This yields $\Sem{\place(a') \otimes \place^{\bot}(1)}{\beta} = \emptyset$ and, because of $\Sem{\place(b')}{\beta} = \emptyset$, also $\Sem{(\place(a') \otimes \place^{\bot}(1)) \Par \place(b')}{\beta} = \emptyset$.
	\end{itemize}
	We conclude that the only possibility is that $a'$ evaluates to $0$ and $b'$ evaluates to $1$.
	Since we can prove these properties in $\EPAomega$, we can transport them via Lemma \ref{lem:embedPeanoIntoLinear} and (dot-eq$_1$) to $\ELPAomegaeq$, and using (dot-sub) we can therefore derive the following from ($\dagger$):
	\begin{multline*}
	\fCenter \exists x^1 (\lp(x) \otimes \place(0) \otimes \emL{A}(x) \otimes \forall y^1 (\lp^{\bot}(y) \Par \place^{\bot}(1) \Par {\emL{B}}^{\bot}(x, y))) \\\Par \forall u^1 (\lp^{\bot}(u) \Par \place^{\bot}(0) \Par {\emL{C}}^{\bot}(u) \Par \exists v^1 (\lp(v) \otimes \place(1) \otimes \emL{D}(u, v)))
	\end{multline*}
	The next step is to invoke Theorem \ref{thm:firstCharacterization}. In order to do this, we apply the following replacement:
	\begin{equation*}
		\place(t) :\equiv t =_0 0 \Par \epsilon =_0 0\period
	\end{equation*}
	By proving similar statements in $\EPAomega$ and transporting them to $\ELPAomegal$ using Lemma \ref{lem:embedPeanoIntoLinear}, we can easily prove the axioms (tag-0), (tag-app), and (tag-con).
	
	For this definition of $\place(t)$, we can derive the two sequents
	\begin{itemize}
		\item $\fCenter \place^{\bot}(0), 1$
		\item $\fCenter \place(0)$
	\end{itemize}
	using (1), (d?), and ($!_2$) for the first sequent, and Lemma \ref{lem:embedPeanoIntoLinear} for the second one. Moreover, we can prove both
	\begin{itemize}
		\item $\fCenter \place^{\bot}(1), \epsilon =_0 0$
		\item $\fCenter (\epsilon =_0 0)^{\bot}, \place(1)$
	\end{itemize}
	using Lemma \ref{lem:embedPeanoIntoLinear} since we have $\place(1) \equiv (1 =_0 0 \Par \epsilon =_0 0)$.
	Then, with both Lemma \ref{lem:substitute} and Corollary \ref{cor:substituteTensorPar}, we can show
	\begin{multline*}
	\fCenter \exists x^1 (\lp(x) \otimes \emL{A}(x) \otimes \forall y^1 (\lp^{\bot}(y) \Par \epsilon \neq_0 0 \Par {\emL{B}}^{\bot}(x, y))) \\\Par \forall u^1 (\lp^{\bot}(u) \Par {\emL{C}}^{\bot}(u) \Par \exists v^1 (\lp(v) \otimes \epsilon =_0 0 \otimes \emL{D}(u, v)))\period
	\end{multline*}
	Finally, we reintroduce all abbreviations
	\begin{equation*}
	\fCenter \forallL x^1 (\emL{A}(x) \lto \existsLeps y^1 \emL{B}(x, y)) \lto \forallL u^1 (\emL{C}(u) \lto \existsLeps v^1 \emL{D}(u, v))
	\end{equation*}
	and apply Theorem \ref{thm:firstCharacterization} to it.
\end{proof}

\section{Counterexamples to ``On Weihrauch Reducibility and Intuitionistic Reverse Mathematics''}

In this section, we will discuss the results developed by Rutger Kuyper in \cite{kuyper2017}. There, he presented two theorems about the characterization of a formalization of Weihrauch reducibility in the system $\RCA_0$, which is the weakest of the big five calculi used in reverse mathematics. Note that the definitions of the following calculi differ from those usually meant in literature.

We start with the sequent calculus $\IQC$ that contains the usual terms, formulas, and rules of intuitionistic logic with the exception of disjunction and all of its rules. For the definition of $\EL_0$, elementary logic with quantifier-free induction, we differentiate between number and function terms, as well as number and function quantifiers. We use small greek letters $\alpha$, $\beta$, \dots, for function variables. Moreover, we add terms for $0$, the successor function $S$, symbols for all definitions of primitive recursive functions, lambda abstraction, and recursion. The axioms include equality axioms, defining axioms for all primitive recursive functions, the successor axioms, quantifier-free induction, axioms for lambda-abstraction and recursion, and the quantifier-free axiom of choice from numbers to numbers.

There are different definitions of this system (cf. \cite[pp.~243--244]{formalSystems} for an early definition that differs from \cite[p.~13]{fujiwara}, which is used today). Kuyper decided to use Dorais' definition of $\EL_0$ in \cite{dorais2014}. In his work, Dorais first defines the calculus $\EL$, which is the intuitionistic variant of $\RCA$. Dorais does not define disjunction as there exists a way of constructing a substitute using existential quantifiers, implications, and conjunctions. With the help of full induction, which $\EL$ provides, one can show that this substitute satisfies the rules of disjunction (cf. 1.3.7 in \cite[p.~21]{troelstra1973}). Dorais defines $\EL_0$ as a subsystem of $\EL$ where full induction is replaced by quantifier-free induction. This type of induction does not suffice to prove the rules of disjunction for our substitute. Kuyper noticed this problem (cf. Remark 2.4 in \cite[p.~1442]{kuyper2017}) but used Dorais' construction regardless.

Finally, we need $\ELexalpha$ where we restrict the contraction-rule of $\IQC$ to formulas not containing function quantifiers, and the weakening-rule to formulas $A$ and $\exists \alpha A$ where $A$ does not contain any function quantifiers. The results of the article are the following two theorems where the superscript $q$ refers to a negative translation:
\begin{theorem}[Theorem 6.4 in \cite{kuyper2017}]\label{kuyperMulti}
	Let $P_i = \forall \alpha_i (A_i(\alpha_i) \to \exists \beta_i B_i(\alpha_i, \beta_i))$ for $i \in \{0, 1\}$. Then the following are equivalent:
	\begin{enumerate}[label=\alph*)]
		\item There are an $n \in \N$ and $e_1, \dots, e_{n+1}$ such that $\RCA_0$ proves that $e_1$, \dots, $e_{n+1}$ witness that $P_0$ Weihrauch-reduces to the composition of $n$ copies of $P_1$.
		\item For $P'_i := \forall \alpha_i (A_i^q (\alpha_i) \to \exists \beta_i B_i^q (\alpha_i, \beta_i))$ we have that $\EL_0 + \MP$ proves $P'_1 \to P'_0$.
	\end{enumerate}
\end{theorem}

\begin{theorem}[Theorem 7.1 in \cite{kuyper2017}]\label{kuyperSingle}
	If $P_i$ and $P'_i$ are defined as before, then the following are equivalent:
	\begin{enumerate}[label=\alph*)]
		\item There are $e_1$ and $e_2$ such that $\RCA_0$ proves that $e_1$, $e_2$ witness that $P_0$ Weihrauch-reduces to $P_1$.
		\item We have that $(\EL_0 + \MP)^{\exists\alpha a}$ proves $P'_1 \to P'_0$.
	\end{enumerate}
\end{theorem}

\subsection{Needing infinitely many answers at once}

\subsubsection*{The counterexample}

Let $G(x)$ be a primitive recursive predicate such that $G(\bar n)$ where $\bar n$ is the term representing $n$ is true for all numbers $n \in \N$ while $\RCA_0$ is not able to prove $\forall x G(x)$.
We define
\begin{flalign*}
	&&P_1 &:= \forall \alpha_1 (0=0 \to \exists \beta_1 G(\alpha_1(0)))&\\
	\text{and} &&P_0 &:= \forall \alpha_0 (0=0 \to \exists \beta_0 \forall x G(x)) \period&
\end{flalign*}
Using $\EL_0$, we can prove $P'_1 \to P'_0$, i.e.
\begin{equation*}
	\forall \alpha_1 (\lnot\lnot 0=0 \to \exists \beta_1 \neg \neg G(\alpha_1(0))) \to \forall \alpha_0 (\lnot\lnot 0=0 \to \exists \beta_0 \neg \neg \forall x \neg \neg G(x)) \period
\end{equation*}
Theorem~\ref{kuyperMulti} tells us that there exists a number $n \in \N$ and $e_1, \dots, e_{n+1} \in \N$ such that $\RCA_0$ proves that $e_1, \dots, e_{n+1}$ witness the Weihrauch reduction from $P_0$ to $n$ copies of $P_1$. We define $\alpha_0$ and $\beta_1, \dots, \beta_n$ to be constant zero-functions and can therefore assume without loss of generality that the programs $e_1, \dots, e_{n+1}$ neither use $\alpha_0$ nor any of $\beta_1, \dots, \beta_n$ as oracles. Let $\Phi_{e}$ be the computable function induced by a program $e$. We will now show that the sequents $C_n \vdash D_n$ with
\begin{align*}
	C_n &:= \bigwedge_{i = 1}^{n} (\halts{\Phi_{e_i}(\alpha_1, \dots, \alpha_{i-1})} \to \alpha_i = \Phi_{e_i}(\alpha_1, \dots, \alpha_{i-1}))\\
	D_n &:= \bigwedge_{i=1}^{n} (\bigwedge_{j=1}^{i-1} G(\alpha_j(0)) \to \halts{\Phi_{e_i}(\alpha_1, \dots, \alpha_{i-1})}) \wedge (\bigwedge_{i=1}^{n} G(\alpha_i(0)) \to \forall x G(x))
\end{align*}
are not derivable in RCA$_0$ for any programs $e_1, \dots, e_n$.
For this, we assume and contradict $C_n \vdash D_n$ by induction on $n \in \N$.

Let $n=0$: In this case all conjunctions and premises collapse, i.e.~$C_0$ is empty and $D_0$ only consists of $\forall x G(x)$. But now, $\RCA_0$ proves $\forall x G(x)$. This is a contradiction to the assumption we made about $G$.

Let $n > 0$: Assume that $\RCA_0$ proves the sequent for $n$ copies of $P_1$ (and programs $e_1, \dots, e_{n+1}$) but not for $n-1$ copies. This implies that $\RCA_0$ proves $C_n \vdash \halts{\Phi_{e_1}}$ (and therefore even $C_n \vdash \alpha_1 = \Phi_{e_1}$) because the case $i=1$ in the first conjunction of $D_n$ means that we can prove the totality of $e_1$ directly. We are now interested in deriving $C_n \vdash G(\alpha_1(0))$: We can calculate an $m \in \N$ such that $\Phi_{e_1}(0) = \bar m$ is true because this is a $\Sigma^0_1$-sentence that we can witness by a natural number in order to make it quantifier-free. Since $\RCA_0$ contains all axioms for the definitions of primitive recursive functions, we can even prove this equality. In a similar fashion, the truth of $G(\bar m)$ implies its provability since $G$ is defined primitive recursively. Moreover, for $e_2, \dots, e_n$ there are programs $\tilde{e}_2, \dots, \tilde{e}_n$ that behave the same, but do not depend on the oracle $\alpha_1$ that is given by the total program $e_1$. Combining these results, deriving $C_n \vdash D_n$ in $\RCA_0$ implies a derivation of $\tilde{C}_{n-1} \vdash \tilde{D}_{n-1}$ with
\begin{align*}
	\tilde{C}_{n-1} &:= \bigwedge_{i = 1}^{n-1} (\halts{\Phi_{\tilde{e}_{i+1}}(\alpha_2, \dots, \alpha_i)} \to \alpha_{i+1} = \Phi_{\tilde{e}_{i+1}}(\alpha_2, \dots, \alpha_i))\\
	\tilde{D}_{n-1} &:= \bigwedge_{i=1}^{n-1} (\bigwedge_{j=1}^{i-1} G(\alpha_{j+1}(0)) \to \halts{\Phi_{\tilde{e}_{i+1}}(\alpha_2, \dots, \alpha_i)}) \wedge (\bigwedge_{i=1}^{n-1} G(\alpha_{i+1}(0)) \to \forall x G(x)) \period
\end{align*}
This means that we can prove $C_{n-1} \vdash D_{n-1}$ if we replace the programs $e_1$, \dots, $e_{n-1}$ by the programs $\tilde{e}_2$, \dots, $\tilde{e}_n$. However, this is a contradiction to our induction hypothesis.

\subsubsection*{The fix}
In private correspondence, Kuyper proposed a slight change of Theorem \ref{kuyperMulti} in order to avoid this problem:

\begin{theorem}[Update of Theorem \ref{kuyperMulti}]\label{kuyperMultiUpdated}
	Let $P_i = \forall \alpha_i (A_i(\alpha_i) \to \exists \beta_i B_i(\alpha_i, \beta_i))$ for $i \in \{0, 1\}$. Then the following are equivalent:
	\begin{enumerate}[label=\alph*)]
		\item There are an $n \in \N$ and $e_1, \dots, e_{n+1}$ such that $\RCA_0$ proves \textbf{that $P_1$ implies} that $e_1$, \dots, $e_{n+1}$ witness that $P_0$ Weihrauch-reduces to the composition of $n$ copies of $P_1$.
		\item For $P'_i := \forall \alpha_i (A_i^q (\alpha_i) \to \exists \beta_i B_i^q (\alpha_i, \beta_i))$ we have that $\EL_0 + \MP$ proves $P'_1 \to P'_0$.
	\end{enumerate}
\end{theorem}
The argument for this change was that some part of his proof (cf. Theorem 5.1 \cite[p.~1455]{kuyper2017}) seems to implicitly assume the truth of $P_1$. Our counterexample vanishes with this updated theorem because, assuming $P_1$, we can Weihrauch-reduce $P_0$ to zero copies of $P_1$. Needing at least one such copy was the integral part of our counterexample. From now on, we will work with this updated theorem.

\subsection{Independence-of-premise for quantifier-free formulas}\label{subsec:problemsWithEL0}

\subsubsection*{The counterexample}

The following counterexample shows that under the assumption that Theorem \ref{kuyperMultiUpdated} is correct, the system $\EL_0$ does not behave like an intuitionistic calculus of arithmetic. Kuyper derives his definition from \cite{dorais2014}, where the calculus $\EL$ is defined without disjunction and $\EL_0$ is then later introduced without adding it explicitly. Replacing disjunction by an equivalent construction using the existential quantifier works in calculi with full induction. This calculus $\EL_0$, however, replaces the axiom by quantifier-free induction. Assuming the correctness of Theorem \ref{kuyperMultiUpdated}, we can now show that this induction principle is not strong enough to yield the usual properties of disjunction. Consider the formulas
\begin{flalign*}
	&&P_1 &:= \forall \alpha_1 (\lnot G(\alpha_1(0)) \to \exists \beta_1 \forall e, x (T(e, e, x) \to T(e, e, \beta_1(e))))& \\
	&& P_0 &:= \forall \alpha_0 (0=0 \to \exists \beta_0 (\lnot G(\alpha_0(0)) \to \forall e', x' (T(e', e', x') \to T(e', e', \beta_0(e')))))\comma&
\end{flalign*}
as well as the axiom
\begin{equation*}
	Q := \exists \beta \forall e, x (T(e, e, x) \to T(e, e, \beta(e))) \period
\end{equation*}
Here, let $G$ be a formula such that $\RCA_0 + Q$ is unable to prove the true statement $\forall x G(x)$. We assume that some program in $\RCA_0$ can Weihrauch-reduce $P_0$ to zero copies of $P_1$ using the additional assumption $Q$ that is consistent with $\RCA_0$ and implies $P_1$. If we add the consistent axiom $\exists x \lnot G(x)$ to $\RCA_0 + Q$, we can still prove the reduction, of course. But now, since $Q$ and $\lnot G(\alpha_0(0))$ hold for some $\alpha_0$, the $\beta_0$ our program has to compute is the halting problem. A computable solution for the halting problem is inconsistent with $\RCA_0$ and, therefore, $\RCA_0 + Q$. This contradiction implies that we need at least one copy of $P_1$ for the Weihrauch reduction. Hence, the reduction yields a program that can compute a witness $\alpha_1(0)$ such that $\lnot G(\alpha_1(0))$ holds. This means that $\RCA_0$ together with the true axiom $Q$ proves $\exists x \lnot G(x)$, which cannot be and therefore leads to a contradiction. Theorem \ref{kuyperMultiUpdated} tells us that $\EL_0 + \MP$ is not able to prove $P_1' \to P_0'$. In an intuitionistic calculus, however, this proof should be trivial: Starting from $P_1$, such a proof would pull the $\exists \beta_1$ in front of $\lnot G(\alpha_1(0))$ via independence of premise for quantifier-free formulas, and weaken the statement by adding the premise $0=0$ in front of the existence quantifier.

\subsubsection*{The fix}

It is not clear how to repair this issue. The notion of realizability used in the proof of Kuyper's theorem might need a substantial change since it seems to be incompatible with giving witnesses for existential statements that heavily depend on already quantifier-free premises.

\subsection{Constructing forbidden contractions}

\subsubsection*{The counterexample}

This is a counterexample to Theorem 7.1
. Let $G$ and $H$ be true sentences of $\RCA_0$ such that this calculus does not prove $G \lor H$, $\lnot G \lor H$, $G \lor \lnot H$, or $\lnot G \lor \lnot H$.\footnote{As pointed out by Ulrich Kohlenbach, the existence of such sentences is guaranteed by Theorem 9 in \cite[p.~31]{lindstroem}, which is based on a result due to Dana Scott in \cite{scott1962}.} We abbreviate
\begin{flalign*}
	&&A(x) &:\equiv x=0 \leftrightarrow G&\\
	\text{and}&&B(x) &:\equiv x=0 \leftrightarrow H\comma&
\end{flalign*}
and define two problems $P_1$ and $P_0$ with an intermediate step $P_{\nicefrac{1}{2}}$:
\begin{flalign*}
	&P_1 &&:= \forall \alpha_1 (0=0 \to \exists \beta_1 ((\alpha_1(0)=0 \to A(\beta_1(0))) \land (\alpha_1(0) \neq 0 \to B(\beta_1(0)))))\comma&\\
	&P_{\nicefrac{1}{2}} &&:= \forall x (0=0 \to \exists y ((x=0 \to A(y)) \land (x \neq 0 \to B(y))))\comma&\\
	&P_0 &&:= \forall \alpha_0 (0=0 \to \exists \beta_0 (A(\beta_0(0)) \land B(\beta_0(1)))) \period&
\end{flalign*}
We can easily prove $P'_1 \to P'_{\nicefrac{1}{2}}$ in $(\EL_0 + \MP)^{\exists \alpha a}$ by replacing the function variables with number variables. For the proof of $P'_{\nicefrac{1}{2}} \to P'_0$, we use $P'_{\nicefrac{1}{2}}$ twice: First with $x=0$ to get some $y$ with $A(y)$, and then with $x=1$ to get some $z$ with $B(z)$. Using recursion, we can define a $\beta_0$ with $\beta_0(0) = y$ and $\beta_0(1) = z$, and yield $P'_0$. By contraction for formulas without function variables, we can contract the two instances of  $P'_{\nicefrac{1}{2}}$ to only one instance. By using the cut-rule, we can combine $P'_1 \to P'_{\nicefrac{1}{2}}$ and $P'_{\nicefrac{1}{2}} \to P'_0$ to a get proof of $P'_1 \to P'_0$.

Now, Theorem \ref{kuyperSingle} says that there are programs $e_1$ and $e_2$ in $\RCA_0$ that witness the Weihrauch reduction from $P_0$ to $P_1$. We consider what this reduction does in the standard model: Let $\alpha_0$ be the constant zero-function, and without loss of generality let $e_1$ compute an $\alpha_1$ with $\alpha_1(0) = 0$. Since $G$, which is the same as $A(0)$, is true, we can assume that $\beta_1$ is the constant zero-function, as well. We know that $H$, which is the same as $B(0)$, is true. Therefore, $e_2$ has to compute a $\beta_0$ with $\beta_0(1) = 0$. We conclude, that given two zero-functions as input, $e_2$ computes a function $\beta_0$ with $\beta_0(1)=0$. Like in the first counterexample, we can prove such true evaluations of programs in $\RCA_0$. If we consistently add $G$ and $\lnot H$ to $\RCA_0$, and define $\alpha_0$ and $\beta_1$ to be constant zero-functions, then the existence-part of $P_1$ holds for the $\alpha_1$ computed by $e_1$ and the Weihrauch reduction tells us that $H$ holds, which it does not by assumption.

\subsubsection*{The fix}

The whole procedure only works because we use the cut-rule as a way to conceal an otherwise impossible contraction from the calculus. The proposed solution by Kuyper is the removal of free cuts from the calculus. This could be a very uncomfortable system to work with, and it has also not been shown yet that the proof of the theorem still works with this updated calculus.

\subsection*{Acknowledgment}
The first part of this paper about the characterization ($\dagger$) establishes a refined and adapted result from my Master's thesis (cf.~\cite{Uft18}) supervised by Ulrich Kohlenbach. Additionally to this supervision, I am grateful for his help during this paper, as well as his early contributions by leading me to the right literature (by Paolo Oliva on functional interpretations of linear logic) and the idea to handle the quantifiers within problems differently from those that ensure the well-definedness of problems. Finally, I am thankful for the helpful suggestions that I received during the review process.

\bibliographystyle{amsplain}
\bibliography{characterization}
\end{document}